\documentclass[12pt,reqno]{amsart}
\usepackage{amssymb} 
\usepackage{verbatim,xcolor}
\usepackage{appendix}
\usepackage{hyperref}

\usepackage[top=35mm, bottom=35mm, left=30mm, right=30mm]{geometry}

\theoremstyle{plain}
\newtheorem{thm}{Theorem}[section]

\newtheorem{lem}[thm]{Lemma}
\newtheorem{prop}[thm]{Proposition}
\newtheorem{conj}[thm]{Conjecture}

\theoremstyle{definition}
\newtheorem{de}[thm]{Definition}

\numberwithin{equation}{section}

\allowdisplaybreaks

\begin{document}
	
	\author{Wen Huang, Leiye Xu and Xiangdong Ye} \address {Wu Wen-Tsun Key Laboratory of Mathematics, USTC, Chinese Academy of Sciences and Department of Mathematics, \\
		University of Science and Technology of China,\\
		Hefei, Anhui, China}
	\email{wenh@mail.ustc.edu.cn,leoasa@mail.ustc.edu.cn, yexd@ustc.edu.cn}

	\title[Complexity and Logarithmic Sarnak conjecture]{Polynomial mean complexity and Logarithmic Sarnak conjecture}
	
	\thanks{}
	
\begin{abstract}In this paper, we reduce the logarithmic Sarnak conjecture to the $\{0,1\}$-symbolic
systems with polynomial mean complexity. By showing that the logarithmic Sarnak conjecture holds for any
topologically dynamical system with sublinear complexity, we provide a variant of the $1$-Fourier
uniformity conjecture, where the frequencies are restricted to any subset of $[0,1]$ with packing dimension less than one.
		
	\end{abstract}

	\maketitle
	
	\parskip 0.4cm
	\section{Introduction}
	In this paper, a {\it topologically dynamical system} (t.d.s. for short) is a pair $(X, T)$, where
	$X$ is a compact metric space endowed with a metric $d$ and $T: X \to X$ is a homeomorphism.
Denote by $\mathcal{M}(X,T)$ the set of all $T$-invariant Borel probability
measures on $X$, which is a non-empty convex and compact metric space with respect to the weak$^*$ topology.
 We say a sequence $\xi$ is {\em realized} in $(X,T)$ if there is an $f\in C(X)$ and an $x\in X$
	such that $\xi(n) = f(T^nx)$ for any $n\in\mathbb{N}$. A sequence $\xi$ is called {\em deterministic} if it is realized in a t.d.s.
 with zero topological entropy. The M\"{o}bius function $\mu: \mathbb{N}\rightarrow \{-1,0,1\}$ is defined by
	$\mu(1)=1$ and
	\begin{equation}\label{M-function}
	\mu(n)=\left\{
	\begin{array}{ll}
	(-1)^k & \hbox{if $n$ is a product of $k$ distinct primes;} \\
	0 & \hbox{otherwise.}
	\end{array}
	\right.
	\end{equation}

In this paper, $\mathbb{N}=\{1,2,\cdots\}$, $\mathbb{E}$ (resp. $\mathbb{E}^{log}$) stands for a
finite average (resp. a finite logarithmical average), i.e., $$\displaystyle\mathbb{E}_{n\le N}A_n= \frac1{N}
\sum_{n=1}^{N}A_n\text{ and } \displaystyle\mathbb{E}^{log}_{n\le N}A_n= \frac1{\sum_{n=1}^N\frac{1}{n}}\sum_{n=1}^{N}\frac{A_n}{n}.$$
Here is the well-known conjecture by Sarnak \cite{Sar}:

	\noindent {\bf Sarnak Conjecture:}\ {\em
		The M\"{o}bius function $\mu$ is linearly asymptotically disjoint from any deterministic sequence $\xi$. That is,
		\begin{equation}\label{Sarnak}
		\lim_{N\rightarrow \infty}\mathbb{E}_{n\le N}\mu(n)\xi(n)=0.
		\end{equation}
	}

	The conjecture in the case when $X$ is finite is equivalent to the prime number theorem in
	arithmetic progressions. And the conjecture in case when T is a rotation on the circle is equivalent to
    the Davenport's theorem \cite{D37}. The conjecture in many other special cases have been established
    more recently (see \cite{GT, FKL2018, FKL2019, KLR} and references therein).

	Tao introduced and investigated the following  logarithmic version of Sarnak conjecture \cite{Tao,Tao1} (see also \cite{FHost,TV,TV1,M}).
	
	\medskip
	
	\noindent {\bf Logarithmic Sarnak Conjecture:}\ {\em
		For any topological
		dynamical system $(X,T)$ with zero entropy, any continuous function $f:X\to\mathbb{C}$
		and any point $x$ in $X$,
		\begin{equation}\label{log-Sarnak}
		\lim_{N\rightarrow \infty}\mathbb{E}_{n\le N}^{log}\mu(n)f(n)=0.
		\end{equation}
	}

	Now we let $(X,T)$ be a t.d.s. with a metric $d$. For any $n\in \mathbb{N}$, we consider the so-called {\it mean metric}
induced by $d$
	$$\overline{d}_n(x,y)=\frac{1}{n}\sum_{i=0}^{n-1}d(T^ix,T^iy)$$
	for any $x,y\in X$. For $\epsilon>0$ and a subset $K$ of $X$, we let
	$$S_n(d, T,K ,\epsilon)=\min \{ m\in \mathbb{N}:\exists x_1,x_2,\cdots,x_m
	\text{ s.t. }K\subset\bigcup_{i=1}^m B_{\overline{d}_n}(x_i,\epsilon)\},$$
	where $B_{\overline{d}_n}(x,\epsilon):=\{y\in X: \overline{d}_n(x,y)<\epsilon\}$
	for any $x\in X$. We say $(X,T)$ has {\it polynomial mean complexity} if there exists a constant
$k>0$ such that $\liminf \limits_{n\to+\infty}\frac{S_n(d,T,X,\epsilon)}{n^k}=0$ for all $\epsilon>0$.
	The following is the our main result.
	\begin{thm}\label{thm-1}
		The following statements are equivalent:
		\begin{enumerate}
			\item The logarithmic Sarnak conjecture holds.
			\item The logarithmic Sarnak conjecture holds for any t.d.s. with polynomial mean complexity.
			\item The logarithmic Sarnak conjecture holds for any $\{0,1\}$-symbolic system with polynomial mean  complexity.
		\end{enumerate}
	\end{thm}

{\color{black} We now briefly describe the main ingredients in the proof of Theorem \ref{thm-1}.
It is clear that (1) implies (2) which in trun implies (3). So it remains to prove
(2) implies (1) and (3) implies (2). To show (2) implies (1), we  use Tao's result as a starting point,
which states that the logarithmic Sarnak conjecture
is equivalent to a conjecture involving the limit of averages on nilmanifolds, see Conjecture \ref{1}.
By assuming that Conjecture \ref{1} fails, we are able then to construct a system  with polynomial mean complexity
which does not satisfy the logarithmic Sarnak conjecture, and hence prove that (2) implies (1).
To construct the system, we need to work on nilsystems and figure out the complexity of polynomial
sequences, see Proposition \ref{lem-33}. Precisely, we will show that for a given $\epsilon>0$, for any $n\in\mathbb{N}$,
the minimal number of $\epsilon$-dense subsets of strings of lengths $n$ of the set of all polynomial
sequences on $G/\Gamma$ is bounded by a polynomial which is only dependent on $\epsilon$ and $G/\Gamma$,
where $G/\Gamma$ is an $s$-step nilmanifold.
With the help of this proposition we finish the construction and thus show that (2) implies (1).
To show (3) implies (2), we study a t.d.s. with the small boundary
property which was introduced by Lindenstrauss when studying mean dimension. Proposition \ref{lem-2.10}
plays a key role for the proof, which states that for a  t.d.s. $(X,T)$ with  polynomial mean  complexity
and a subset $U$ with small boundary, each $x\in X$ is associated with a point in the shift space
such that the complexity of the closure of the associated points is less than or equals to
that of $(X,T)$. 
The result
of Lindenstrauss and Weiss guarantees that if $(X,T)$ has zero entropy then the prodcut of $X$ with
any irrational rotation on the circle has the small boundary property. By using Proposition \ref{lem-2.10}
and some simple argument we finish the proof that (3) implies (2), and hence the proof of Theorem \ref{thm-1}.

\medskip
While Theorem \ref{thm-1} does not provide a proof of the logarithmic Sarnak conjecture directly, it does
indicate that a t.d.s. with polynomial mean complexity is important for the proof of the conjecture.
So, it will be useful to understand the structure of a subshift with polynomial mean complexity.
We remark that we do not know if the polynomial mean complexity for a subshift can be replaced by the polynomial
block-complexity in Theorem \ref{thm-1}, which is extensively studied in the literature.}

\medskip
	
For a t.d.s. $(X,T)$ with a metric $d$, $\epsilon>0$ and a $\rho\in {\mathcal{M}}(X,T)$, we let
$$S_n(d,T,\rho,\epsilon)=\min \{ m\in \mathbb{N}:\exists x_1,x_2,\cdots,x_m
\text{ s.t. } \rho\big(\bigcup_{i=1}^m B_{\overline{d}_n}(x_i,\epsilon)\big)>1-\epsilon\}.$$

It is clear that $S_n(d,T,\rho,\epsilon)\le S_n(d, T,X ,\epsilon)$ for any $\rho\in {\mathcal{M}}(X,T)$ and $\epsilon>0$.
We say $(X,T)$ has {\it sub-linear mean measure complexity} if   $$\liminf_{n\rightarrow +\infty}
 \frac{S_n(d,T,\rho,\epsilon)}{n}=0$$ for any $\epsilon>0$ and any $\rho\in\mathcal{M}(X,T)$.
By using the fact that the two-terms logarithmic Chowla conjecture holds \cite{Tao}, i.e.
\begin{align}\label{Tao}\lim_{N\rightarrow \infty}\frac{1}{\ln N}\sum_{n=1}^{N}\frac{\mu(n+h_1)\mu(n+h_2)}{n}=0
\end{align}
for any $0\le h_1<h_2\in \mathbb{N}$, and by using the method of the proof of Theorem 1.1' in \cite{HWY} we have

\begin{thm}\label{rem-linear} The  logarithmic Sarnak conjecture holds for any t.d.s. with sub-linear mean measure complexity.
Consequently, the conjecture holds for any t.d.s. with sub-linear mean complexity.
\end{thm}

{\color{black}
We remark that at this moment we are not able to show that the logarithmic Sarnak conjecture
holds for any t.d.s. with linear mean (measure) complexity.
We also remark that if for any $k\in \mathbb{N}$ the $2k$-term logarithmic Chowla conjecture holds, i.e.
\begin{align}\lim_{N\rightarrow \infty}\frac{1}{\ln N}\sum_{n=1}^{N}\frac{\mu(n+h_1)\mu(n+h_2)\ldots \mu(n+h_{2k})}{n}=0\end{align}
for any non-negative integer $0\le h_1\le h_2\le \ldots \le h_{2k}$ with an odd number $j\in \{1,2,\cdots,2k\}$ such that $h_j<h_{j+1}$, then the logarithmic Sarnak conjecture holds for any t.d.s. with sub-polynomial
(leading term $cn^k$) mean measure complexity by using the method of Theorem \ref{rem-linear}. Thus, by Theorem \ref{thm-1} we know that
the logarithmic Sarnak conjecture holds if the logarithmic Chowla conjecture holds. In fact, the two conjectures are equivalent \cite{Tao}.
}


	As an application of Theorem \ref{rem-linear}, one has the following result.
	\begin{thm} \label{thm-2}  Let $C$ be a non-empty compact subset of $[0,1]$ with packing dimension $<1$. Then
		\begin{align}\label{eq-1}
		\lim\limits_{H\rightarrow +\infty}\limsup\limits_{N\rightarrow +\infty} \mathbb{E}_{n\le N}^{log} \sup_{\alpha\in C}
		|\mathbb{E}_{h\le H}\mu (n + h)e(h\alpha)| = 0,
		\end{align}
    where $e(t):=e^{2\pi i t}$ for any $t\in \mathbb{R}$.	
    \end{thm}
	We remark that in \cite[Theorem 1.13]{M}, McNamara proved that \eqref{eq-1}
		holds for  a non-empty compact subset $C$ of $[0,1]$ with upper box dimension $< 1$. 	So Theorem \ref{thm-2} strengthenes
the result in \cite{M}.

We say  a t.d.s. $(X,T)$ has {\it sub-polynomial mean measure complexity} if for any $\tau>0$ and $\rho\in \mathcal{M}(X,T)$
$$\liminf \limits_{n\rightarrow +\infty} \frac{S_n(d,T,\rho,\epsilon)}{n^\tau}=0$$
	for any $\epsilon>0$. In \cite{HWY}, Huang, Wang and Ye showed that the
Sarnak conjecture holds for any t.d.s. with sub-polynomial mean  measure complexity.
As an application of the above result in \cite{HWY}, one has the following result.
	\begin{thm} \label{thm-5}  Let $C$ be a non-empty compact subset of $[0,1]$ with packing dimension $=0$. Then
		\begin{align}\label{eq-2}
		\lim\limits_{H\rightarrow +\infty}\limsup\limits_{N\rightarrow +\infty} \mathbb{E}_{n\le N} \sup_{\alpha\in C}
		|\mathbb{E}_{h\le H}\mu (n + h)e(h\alpha)| = 0.
		\end{align}
	\end{thm}
	\medskip
	The paper is organized as follows. In Section 2, we prove Theorem \ref{thm-1}.  In Section 3, we prove Theorem \ref{thm-2}.
In Appendix \ref{Appendix-A} and \ref{Appendix-B}, we prove Theorem \ref{rem-linear} and Theorem \ref{thm-5}.
	
	\section{Proof of Theorem \ref{thm-1} }
	In this section, we prove Theorem \ref{thm-1}. As we said in the introduction,
it remains to prove $(2)\Longrightarrow(1)$ which is done in subsection \ref{se-1}, and
$(3)\Longrightarrow(2)$ which is carried out in subsection \ref{se-2}.
	
\subsection{Proof of (2) implies (1) in Theorem \ref{thm-1}}\label{se-1}
   We have explained in the introduction that the starting point of the proof is the Tao's
   result which gives an equivalent statement of the logarithmic Sarnak  conjecture. We will first
   introduce the result, then derive some result concerning the complexity of polynomial sequences
   and finally give the proof. Let us begin with basic notions related to nilmanifolds.
	
	Let $G$ be a group. For $g, h\in G$, we write $[g, h] =
	ghg^{-1}h^{-1}$ for the commutator of $g$ and $h$ and we write
	$[A,B]$ for the subgroup spanned by $\{[a, b] : a \in A, b\in B\}$.
	The commutator subgroups $G_j$, $j\ge 1$, are defined inductively by
	setting $G_1 = G$ and $G_{j+1} = [G_j ,G]$. Let $s \ge 1$ be an
	integer. We say that $G$ is {\it $s$-step nilpotent} if $G_{s+1}$ is
	the trivial subgroup.
	
	Recall that {\it an $s$-step nilmanifold} is a manifold of the form
	$G/\Gamma$ where $G$ is a connected, simply connected $s$-step nilpotent Lie group, and $\Gamma$
is a cocompact discrete subgroup of $G$. Tao shows that the logarithmic Sarnak  conjecture  is
equivalent to the following  conjecture \cite{Tao1}.
	\begin{conj}\label{1} For any $s\in\mathbb{N}$, an $s$-step nilmanifold $G/\Gamma$, a
$Lip$-continuous function $F:G/\Gamma\to\mathbb{C}$ and $x_0\in G/\Gamma$, one has
	 $$\lim_{H\to+\infty}\limsup_{N\to+\infty}\mathbb{E}_{n\le N}^{log}\sup_{g\in G}|\mathbb{E}_{h\le H}\mu(n+h)F(g^hx_0)|=0.$$
	\end{conj}
	
	Let $G/\Gamma$ be an $m$-dimensional nilmanifold (i.e. $G$ is a connected, simply connected 	
   $s$-step nilpotent Lie group with unit element $e$ and $\Gamma$ is a cocompact discrete
	subgroup of $G$) and let $G=G_1\supset \ldots\supset G_s\supset
	G_{s+1}=\{e\}$ be the lower central series filtration. 	We will make use of the Lie algebra $\mathfrak{g}$ over
	$\mathbb{R}$ of $G$ together with the exponential map $\exp: \mathfrak{g}\rightarrow G$. Since $G$ is a
     connected, simply connected $s$-step nilpotent Lie group, the exponential map is a diffeomorphism
	\cite{CG, M}. 	A basis $\mathcal {X}= \{X_1, \ldots ,X_m\}$ for the Lie algebra $\mathfrak{g}$ over
	$\mathbb{R}$ is called a {\em Mal'cev basis} for $G/\Gamma$ if the following
	four conditions are satisfied:
	
	\begin{enumerate}
		\item  For each $j = 0,\ldots,m-1$ the subspace $\eta_j := \text{Span}(X_{j+1}, \ldots
		,X_m)$ is a Lie algebra ideal in $\mathfrak{g}$, and hence $H_j := \exp\
		\eta_j$ is a normal Lie subgroup of $G$.
		
		\item  For every $0< i\le s$, there is $l_{i-1}$ such that $G_i = H_{l_{i-1}}$. Thus
			$0=l_0<l_1<\ldots<l_{s-1}\le m-1$.
		
		\item Each $g\in G$ can be written uniquely as
		$\exp(t_1X_1) \exp(t_2X_2)\ldots \exp(t_mX_m)$, for some $t_i\in \mathbb{R}$.
		
		\item $\Gamma$ consists precisely of those elements which, when written in
		the above form, have all $t_i\in \mathbb{Z}$.
	\end{enumerate}

Note that such a basis exists \cite{CG, GT, Mal}. Now we fix a Mal'cev basis $\mathcal {X}= \{X_1, \ldots ,X_m\}$
of $G/\Gamma$. Define $\psi: G\rightarrow \mathbb{R}^m$ such that if $g=\exp(t_1X_1)\cdots\exp(t_mX_m)\in G$, then
$$\psi(g)=(t_1,\cdots,t_m)\in\mathbb{R}^m.$$
Moreover, let $|\psi(g)|=\max_{1\le i\le m}|t_i|$. The following metrics on $G$ and $G/\Gamma$ are introduced in \cite{GT}.

\begin{de} We define $d: G \times G\rightarrow \mathbb{R}$ to be the largest metric
such that $d(x, y)\le |\psi(xy^{-1})|$ for all $x, y\in G$. More explicitly,
we have
$$d(x, y) = \inf\Big\{\sum_{i=1}^n\min\{|\psi(x_{i-1}x_i^{-1})|,
|\psi(x_ix_{i-1}^{-1})|\}: x_0,\ldots, x_n\in G;
x_0=x,x_n=y\Big\}.$$ This descends to a metric on $G/\Gamma$ by
setting
$$d(x\Gamma,y\Gamma):=\inf\{d(x', y'): x', y'\in G; x'=x\ (\text{mod}\ \Gamma);
y'=y\ (\text{mod}\ \Gamma)\}.$$ It turns out that this is indeed a
metric on $G/\Gamma$ (see \cite{GT}). Since $d$ is
right-invariant (i.e., $d(x,y)=d(xg,yg)$ for all $x,y,g\in G$), we also have
$$d(x\Gamma, y\Gamma) = \inf_{\gamma\in \Gamma}d (x, y\gamma).$$
\end{de}

The following lemma appears in \cite[Lemma 7.5 and Lemma 7.6]{DDMSY}.
	\begin{lem}\label{lem-estimate-nil}
Let $G$ be 	a connected, simply connected $s$-step nilpotent Lie group. Then there exist real
polynomials $P_1:\mathbb{R}^3\rightarrow \mathbb{R}$, $P_2:\mathbb{R}\rightarrow \mathbb{R}$ and
$P_3:\mathbb{R}^2\rightarrow \mathbb{R}$ with positive coefficients such that for $x,y,g,h\in G$
		\begin{enumerate}
			\item $d(gx,gy)\le P_1(|\psi(g)|,|\psi(x)|,|\psi(y)|)d(x,y)$;
			\item $|\psi(g^n)|\le P_2(n)|\psi(g)|^{n_G}$, where $n_G$ is a positive constant determined by $G$;
			\item $|\psi(gh)|\le P_3(|\psi(g)|,|\psi(h)|)$.
		\end{enumerate}
	\end{lem}

Let $G$ be 	a  connected, simply connected $s$-step nilpotent Lie group with unit element $e$
and $G= G_0= G_1$, $G_{i+1}=[G,G_i]$ be the lower central series filtration of $G$. It is clear that $\{e\}=G_{s+1}=G_{s+2}=\cdots$.	
By a  {\it polynomial sequence adapted to the lower central series filtration}  we mean a
map ${\bf g} : \mathbb{Z}\to G$ such that $\partial_{h_i},\cdots \partial_{h_1} {\bf g}\in G_i$ for all $i>0$
and $h_1,\cdots ,h_i\in \mathbb{Z}$, where
$$\partial_h {\bf f}(n) := {\bf f}(n+h){\bf f}(n)^{-1}$$
for any map ${\bf f} : \mathbb{Z}\to G$ and $n,h\in \mathbb{Z}$.
Let $Poly(G)$ be the collection
of all polynomial sequences of $G$ adapted to the lower central series filtration.
It is well known that a polynomial sequence ${\bf g} : \mathbb{Z}\to G$ adapted to the lower central series filtration
has unique Taylor coefficients $g_j \in G_j$ for each $0\le j \le s$ such that
	$${\bf g}(n)=g_0^{\tbinom{n}{0}}g_1^{\tbinom{n}{1}}\cdots g_s^{\tbinom{n}{s}},$$
	where $\tbinom{n}{0}\equiv 1$ (see for example \cite[Lemma B.9]{GTZ} and \cite[P.240 Theorem 8]{HK}).
    In this case we say that $g_i\in G_i$
	for $i=0,1,\cdots,s$ is the coefficients of ${\bf g}$.

	Using Lemma \ref{lem-estimate-nil} (2) and (3), it is not hard to verify by induction that there exists a real
polynomial $Q:\mathbb{R}^{s+2}\rightarrow \mathbb{R}$ with positive coefficients such that
	\begin{align}\label{plynomial_Q}
	|\psi({\bf g}(n))|\le Q(n,|\psi(g_0)|,\cdots,|\psi(g_s)|)
	\end{align}
	for $n\in \mathbb{Z}_+$.

We note that for $g,h\in G$, ${\bf g} : \mathbb{Z}\to G$ defined by
	${\bf g}(n)=g^nh$ for each $n\in\mathbb{N}$ is a polynomial sequence adapted to the lower
central series filtration since
$${\bf g}(n)=g^nh=h^{\tbinom{n}{0}}(h^{-1}gh)^{\tbinom{n}{1}}.$$
	
	For a non-empty subset $K$ of $G$, we say  ${\bf g}\in Poly(G)$ {\it a polynomial
sequence with coefficients in $K$}, if 	$g_i\in G_i\cap K$ for $i=0,1,\cdots, s$,
where $\{g_i\}_{i=0}^s$ are the coefficients  of ${\bf g}$. Green, Tao and Ziegler
proved the following lemma (see \cite[Lemma C.1]{GTZ} and \cite[P. 243 Proposition 12]{HK}).
\begin{lem}\label{lem-2-5}Let $G$ be a connected, simply connected $s$-step
nilpotent Lie group and $\Gamma$ be a cocompact discrete subgroup of $G$.  Then
there exists a compact subset $K$ of $G$ such that any polynomial sequence ${\bf g}\in  Poly(G)$ can be factorised as
${\bf g}={\bf g}'{\bf \gamma}$,  where ${\bf g}'\in  Poly(G)$  is a polynomial sequence with coefficients in $K$
and ${\bf \gamma}\in  Poly(G)$ is a polynomial sequence  with coefficients in $\Gamma$.
\end{lem}

Let $X$ be a separable metric space with metric $d$ and $Y$ be a nonempty subset of $X^\mathbb{Z}$.
For any $\epsilon>0$, we let $s_n(Y,\epsilon)$ be the minimal number such that there exist
$x_i\in Y,1\le i\le s_n(Y,\epsilon)$ satisfying that for any $y\in Y$, there exists $1\le i\le s_n(Y,\epsilon)$
with $d(x_i(k),y(k))<\epsilon$ for all $0\le k\le n-1$. Roughly speaking, $s_n(Y,\epsilon)$ is the minimal number
of points which are $\epsilon$-dense in $Y[0,n-1]=\{(y_0,\ldots,y_{n-1}): y=(y_i)_{i\in \mathbb{Z}}\in Y\}$.

Let $G$ be 	a  connected, simply connected $s$-step nilpotent Lie group and  $G/\Gamma$ be an $s$-step
nilmanifold.  For $K\subset G$,  let $Poly(K)$ be the collection of all polynomial sequences adapted to the
lower central series filtration with coefficients in $K$. The map $\pi: Poly(G) \to \{G/\Gamma\}^\mathbb{Z}$ is defined by
$$\pi({\bf g})(n)={\bf g}(n)\Gamma\text{ for all }n\in\mathbb{Z}.$$
Put $Poly(G/\Gamma)=\pi(Poly(G))$. We have the following.
\begin{prop}\label{lem-33} Let $G/\Gamma$ be an $s$-step nilmanifold.  Then there exists
$k\in \mathbb{N}$ such that for each $\epsilon>0$, we find $C(\epsilon)>0$ satisfying
$s_n(Poly(G/\Gamma),\epsilon)\le C(\epsilon)n^k$ for all $n\in \mathbb{N}$.
\end{prop}
To prove Proposition \ref{lem-33}, we need the following lemma.
\begin{lem}\label{lem-44}
	Let $G$ be a connected, simply connected $s$-step nilpotent Lie group and $K$ be a nonempty
compact subset of $G$. Then there is a real polynomial $P: \mathbb{R}\rightarrow \mathbb{R}$ such
that  $$d({\bf g}(n),\widetilde {\bf g}(n))\le P(n)\max\{d(g_i,\widetilde g_i):0\le i\le s\} \text{ for all }n\in \mathbb{N},$$
	for any polynomials ${\bf g}(n)=g_0^{\tbinom{n}{0}}g_1^{\tbinom{n}{1}}\cdots g_s^{\tbinom{n}{s}}$
and $\widetilde {\bf g}(n)=\widetilde g_0^{\tbinom{n}{0}}  \widetilde g_1^{\tbinom{n}{1}}\cdots \widetilde
g_s^{\tbinom{n}{s}}$ adapted to the lower central series filtration with coefficients
$g_0,g_1,\cdots,g_s,\widetilde g_0,\widetilde g_1,\cdots,\widetilde g_s\in K$.
	\end{lem}
\begin{proof} Let $P_1,P_2,P_3$ be the real polynomials appearing in Lemma \ref{lem-estimate-nil}
and $Q$ be the real polynomial appearing in \eqref{plynomial_Q}. Since $K$ is compact,
$w=\max\{|\psi(g)|:g\in K\}$ is a positive real number. Put $\widetilde Q(n)=Q(n,w,w,\cdots,w)$ and
$\widetilde P_2(n)=w^{n_G}P_2(n)$, where $n_G$ is the constant appearing in Lemma \ref{lem-estimate-nil} (2).
	
Let ${\bf g}(n)=g_0^{\tbinom{n}{0}}g_1^{\tbinom{n}{1}}\cdots g_s^{\tbinom{n}{s}}$
 and $\widetilde {\bf g}(n)=\widetilde g_0^{\tbinom{n}{0}}  \widetilde g_1^{\tbinom{n}{1}}\cdots \widetilde  g_s^{\tbinom{n}{s}}$
 be two polynomials adapted to the lower central series filtration with coefficients $g_0,\cdots,g_s,\widetilde g_0,\cdots,\widetilde g_s\in K$.	
 A simple computation yields
	{\small \begin{align}\label{long}\begin{split}
		d({\bf g}(n),\widetilde {\bf g}(n))&\le \sum_{i=0}^{s-1}d\left(g_0^{\tbinom{n}{0}}
\cdots g_{i-1}^{\tbinom{n}{i-1}}\widetilde g_{i}^{\tbinom{n}{i}}\cdots
\widetilde g_s^{\tbinom{n}{s}},g_0^{\tbinom{n}{0}}\cdots g_{i}^{\tbinom{n}{i}}
\widetilde g_{i+1}^{\tbinom{n}{i+1}}\cdots \widetilde g_s^{\tbinom{n}{s}}\right)\\
		&=\sum_{i=0}^{s-1}d\left(g_0^{\tbinom{n}{0}}\cdots g_{i-1}^{\tbinom{n}{i-1}}
\widetilde g_{i}^{\tbinom{n}{i}},g_0^{\tbinom{n}{0}}\cdots g_{i}^{\tbinom{n}{i}}\right)\\
		&\le \sum_{i=0}^{s-1}P_1(|\psi(g_0^{\tbinom{n}{0}}\cdots g_{i-1}^{\tbinom{n}{i-1}})|,
|\psi(\widetilde g_{i}^{\tbinom{n}{i}})|,|\psi(g_{i}^{\tbinom{n}{i}})|)d
\left(\widetilde g_{i}^{\tbinom{n}{i}},g_{i}^{\tbinom{n}{i}}\right)\\
		&\le \sum_{i=0}^{s-1}P_1(\widetilde Q(n),\widetilde P_2(\tbinom{n}{i}),
\widetilde P_2(\tbinom{n}{i}))d\left(\widetilde g_{i}^{\tbinom{n}{i}},g_{i}^{\tbinom{n}{i}}\right)\\
		&\le \widetilde P(n)\sum_{i=0}^{s-1}d\left(\widetilde g_{i}^{\tbinom{n}{i}},g_{i}^{\tbinom{n}{i}}\right)
		\end{split}
		\end{align}}
	for all $n\in \mathbb{N}$, where $\widetilde P(n)=\sum_{i=0}^{s-1}P_1(\widetilde Q(n),\widetilde
P_2(\tbinom{n}{i}),\widetilde P_2(\tbinom{n}{i}))$ is a polynomial of $n$.

	Now we are going to show that  there is a real polynomial $P_4: \mathbb{R}\rightarrow \mathbb{R}$ such
that  $d\left(\widetilde g^{n},g^{n}\right)\le P_4(n)d\left(\widetilde g,g\right)$ for all $g,\widetilde g\in K$. In fact, it
follows from the fact
	\begin{align}\label{long-1}\begin{split}
	d\left(\widetilde g^{n},g^{n}\right)&\le \sum_{i=0}^{n-1}d(\widetilde g^ig^{n-i},\widetilde g^{i+1}g^{n-i-1})=\sum_{i=0}^{n-1}d(\widetilde g^ig,\widetilde g^{i+1})\\
	&\le\sum_{i=0}^{n-1}P_1(|\psi(\widetilde g^i)|,|\psi(\widetilde g)|,|\psi(g)|)d\left(\widetilde g,g\right)\\
	&\le\sum_{i=0}^{n-1}P_1(\widetilde P_2(i),w,w)d\left(\widetilde g,g\right)\\
	&\le P_4(n)d\left(\widetilde g,g\right)
	\end{split}
	\end{align}
	for all $n\in \mathbb{N}$, where $P_4(n)=\sum_{i=0}^{n-1}P_1(\widetilde P_2(i),w,w)$ is a real polynomial of $n$. Summing up,
we obtain
	\begin{align*}
	d({\bf g}(n),\widetilde {\bf g}(n))&\overset{\eqref{long}}\le  \widetilde P(n)\sum_{i=0}^{s-1}d\left(\widetilde g_{i}^{\tbinom{n}{i}},g_{i}^{\tbinom{n}{i}}\right)\\
	&\overset{\eqref{long-1}}\le  \widetilde P(n)\sum_{i=0}^{s-1}P_4(\tbinom{n}{i})d\left(\widetilde g_i,g_i\right)\\
	&\le P(n)\max\{d(g_i,\widetilde g_i):0\le i\le s\}
	\end{align*}
	for all $n\in \mathbb{N}$, where $P(n)=\widetilde P(n)\sum_{i=0}^{s-1}P_4(\tbinom{n}{i})$ is a real polynomial of $n$.
Then $P(n)$ is the real polynomial as required.
	This ends the proof of Lemma \ref{lem-44}.
	\end{proof}
Now we are ready to prove Proposition \ref{lem-33}.
\begin{proof}[Proof of Proposition \ref{lem-33}] By Lemma \ref{lem-2-5}, there exists a compact subset $K$ of $G$ such that
any polynomial sequence ${\bf g}$ adapted to the lower central series filtration can be factorised as
		${\bf g}={\bf g}'{\bf \gamma}$,
		where ${\bf g}'$ is a polynomial sequence adapted to the lower central series filtration with coefficients in $K$
		and ${\bf \gamma}$ is a polynomial sequence with coefficients in $\Gamma$.  Since $K$ is compact, by Lemma \ref{lem-44},
there is a real polynomial $P: \mathbb{R}\rightarrow \mathbb{R}$ such that  \begin{align}\label{distance}
		d({\bf g}(j),\widetilde {\bf g}(j))\le P(j)\max\{d(g_i,\widetilde g_i):0\le i\le s\} \text{ for all }j\in \mathbb{N},\end{align}
		and any polynomials ${\bf g}$, $\widetilde {\bf g}\in Poly(G)$ with coefficients
$g_0,\cdots,g_s,\widetilde g_0,\cdots,\widetilde g_s\in K$. It is not hard to see that there exists  $k_0\in\mathbb{N}$ and $C>1$ such that
		\begin{align}\label{C}
		 P(n)<Cn^{k_0}\text{ for all }n\in\mathbb{N}.
		\end{align}
			Since $K$ is compact, for
		$\epsilon > 0$, we let $N_\epsilon(K)$ be the smallest number of open balls of ratio $\epsilon $ needed to cover $K$. The upper
		Minkowski dimension or box dimension (see \cite{M95}) is defined by
		$$\limsup_{\epsilon\to0}\frac{-\log N_\epsilon(K)}{\log\epsilon}.$$
		This dimension of $K$ is not larger than the usual dimension of $G$ since $K$ is a subset of $G$.
Hence there exists a positive constant $L$ such that \begin{align}\label{dim}
		N_\epsilon(K)\le L(\frac{1}{\min\{ \epsilon,1\}})^{dim(G)+1}.
		\end{align}	
		
		 Set
		 $$k=k_0(s+1)(dim(G)+1)\text{ and }C(\epsilon)=\left(L(\frac{2C}{\min\{ \epsilon,1\}})^{dim(G)+1}\right)^{s+1}\text{ for }\epsilon>0.$$

		 We are going to show that $$s_n(G/\Gamma,\epsilon)\le C(\epsilon) n^k$$ for $n\in\mathbb{N}$ and $\epsilon>0$. To do this, let
$\pi$ be the projection from $Poly(K)$ to $Poly(G/\Gamma)$ defined by $\pi({\bf g})(n)={\bf g}(n)\Gamma$ for all
$n\in\mathbb{Z}$. By Lemma \ref{lem-2-5}, $\pi$ is surjective  and
		 $$d({\bf g}(j),{\bf\widetilde g}(j))\ge d(\pi({\bf g})(j),\pi({\bf \widetilde g})(j))\text{ for all }j\in\mathbb{Z}.$$
		 Hence
		 \begin{align}\label{poly1}s_n(Poly(G/\Gamma),\epsilon)\le s_n(Poly(K),\epsilon)\text{ for all }n\in\mathbb{N} \text{ and }\epsilon>0.
		 \end{align}
For $\tau>0$, we let $E_{\tau}$ be a finite subset of $K$ such that
\begin{align*}
\sharp E_\tau\le N_{\tau}(K)\text{ and }K\subset \bigcup_{g\in E_\tau}B(g,\tau).
\end{align*}
For $0\le i\le s$, we let $E_{\tau}^{(i)}$ be a subset of $K\cap G_i$ such that
\begin{align}\label{N}
\sharp E_\tau^{(i)}\le N_{\tau}(K)\text{ and }K\cap G_i\subset \bigcup_{g\in E^i_\tau}B(g,2\tau).
\end{align}
Put $P_\tau$ be the collection of all polynomial sequences ${\bf g}$ adapted to the lower central series filtration with
coefficients $g_i\in E_\tau^{(i)}$, $i=0,1,\cdots,s$.  Then  for $n\in\mathbb{N}$ and $\epsilon>0$
\begin{align}\label{poly2}\sharp P_{\frac{\epsilon}{2Cn^{k_0}}}=
\prod_{i=0}^s\sharp E^{(i)}_{\frac{\epsilon}{2Cn^{k_0}}}\overset{\eqref{N},\eqref{dim}}
\le \left(L(\frac{2Cn^{k_0}}{\min\{ \epsilon,1\}})^{dim(G)+1}\right)^{s+1}=C(\epsilon)n^k.
\end{align}
Now we fix $n\in\mathbb{N}$ and $\epsilon>0$.  By \eqref{N}, for any polynomial sequence
${\bf g}\in Ploy(K)$  with coefficients $g_0,\cdots,g_s\in K$, we have that $g_i\in K\cap G_i$.
Thus, there exists ${\bf \bar g}\in P_{\frac{\epsilon}{2Cn^{k_0}}}$ with coefficients
$\bar g_0\in E^{(0)}_{\frac{\epsilon}{2Cn^{k_0}}} ,\cdots,\bar g_s\in E^{(s)}_{\frac{\epsilon}{2Cn^{k_0}}}$ such that
$$d(g_i,\bar g_i)\overset{\eqref{N}}<\frac{\epsilon}{Cn^{k_0}}\text{ for all }0\le i\le s.$$
Therefore, $$d({\bf g}(0),\bar{\bf g}(0))=d(g_0,\bar{g}_0)<\frac{\epsilon}{Cn^{k_0}}<\epsilon$$
and for $1\le j\le n-1$, and one has
$$d({\bf g}(j),\bar{\bf g}(j))\overset{\eqref{distance}}\le P(j)
\max\{d(g_i,\bar g_i):0\le i\le s\}\overset{\eqref{C}}< Cj^{k_0}\times\frac{\epsilon}{Cn^{k_0}}\le \epsilon.$$
Hence,
$$s_n(Poly(G/\Gamma),\epsilon)\overset{\eqref{poly1}}\le s_n(Poly(K),\epsilon)
\le \sharp P_{\frac{\epsilon}{2Cn^{k_0}}}\overset{\eqref{poly2}}\le C(\epsilon)n^k.$$
Since the above inequality holds for all $n\in\mathbb{N}$ and $\epsilon>0$, we end the proof of Proposition \ref{lem-33}.
\end{proof}
	
With the above preparations, now we are in the position to prove Theorem \ref{thm-1}.

\begin{proof}[Proof of (2)$\Longrightarrow$(1) in Theorem \ref{thm-1}] Assume that Theroem \ref{thm-1} (2)
holds, that is, the logarithmic Sarnak conjecture holds for any t.d.s. with polynomial mean complexity.
In the sequel we aim to show that the logarithmic Sarnak conjecture holds.

Assume the contrary that this is not the case,  then by Tao's result \cite{Tao1} the Conjecture \ref{1} does not hold.
This means that there exist an $s\in\mathbb{N}$, an $s$-step nilmanifold $G/\Gamma$, a $Lip$-continuous function
$F:G/\Gamma\to\mathbb{C}$ and an $x_0\in G/\Gamma$ such that
		\begin{align}\label{78-0}
		\limsup_{H\to+\infty}\limsup_{N\to+\infty}\mathbb{E}_{n\le N}^{log}\sup_{g\in G}\big| \mathbb{E}_{h\le H}\mu(n+h)F(g^hx_0)\big|>0.	
		\end{align}
It is clear that   $\|F\|_{\infty}:=\max_{x\in G/\Gamma}|F(x)|>0$. Without loss of generality, we assume that
    \begin{align}\label{f-infty}
	\|F\|_\infty=1.
    \end{align}
Now we add an extra point $p$ to the compact metric space $G/\Gamma$. We then extend the metric $d$ on $G/\Gamma$ to
the space $G/\Gamma\cup\{p\}$ by let $d(p,x)=1$ for all $x\in G/\Gamma$. So, $(G/\Gamma\cup\{p\},d)$ is also
a compact metric space. Let $\widetilde F: \left(G/\Gamma\cup\{p\}\right)^{\mathbb{Z}}\to \mathbb{C}$ be
defined by $\widetilde F(z)=F(z(0))$ if $z(0)\in G/\Gamma$ and $0$ if $z(0)=p$.  It is clear that $\widetilde F$ is a continuous function and
\begin{align}\label{wF}\|\widetilde F\|_\infty =1
	\end{align}
by \eqref{f-infty}.
		
	In the sequel, we will find a point $y\in \left(G/\Gamma\cup\{p\}\right)^{\mathbb{Z}}$ such that
		\begin{align}\label{equation}\limsup_{N\to\infty}|\mathbb{E}_{n\le N}^{log}
		\mu (n)\widetilde F(\sigma^ny)|>0,\end{align}
		and the t.d.s.
		$(X_y,\sigma)$ has polynomial mean complexity, where  $\sigma: \left(G/\Gamma\cup\{p\}\right)^{\mathbb{Z}}\rightarrow
\left(G/\Gamma\cup\{p\}\right)^{\mathbb{Z}}$ is the left shift and $X_y=\overline{\{\sigma^ny:n\in\mathbb{Z}\}}$  is a $\sigma$-invariant
		compact subset of $\left(G/\Gamma\cup\{p\}\right)^{\mathbb{Z}}$. Clearly, this is a contradiction to
our assumption and thus proves that (2) implies (1) in Theorem \ref{thm-1}.

We divide the remaining proof into two steps.
		
		\medskip
		{\it \noindent {\bf Step 1.} The construction of the point $y$.}
		
		Firstly, we note that
		$$|z|\le \sum_{j=0}^{3} \max\{ Re(e(\frac{j}{4})z),0\}$$ for $z\in \mathbb{C}$. Thus by \eqref{78-0},
there is $\beta\in \{0,\frac{1}{4},\frac{2}{4},\frac{3}{4}\}$ such that
		\begin{align*}\limsup_{H\to+\infty}\limsup_{N\to+\infty}\mathbb{E}_{n\le N}^{log}\max\{ \sup_{g\in G}
Re\left( e(\beta)\mathbb{E}_{h\le H}\mu(n+h)F(g^hx_0)\right),0\}>0.\end{align*}
	    Thus we can find  $\tau\in(0,1)$ with
	    \begin{align*}
	    	E:=\{H\in \mathbb{N}: \limsup_{N\to+\infty}\mathbb{E}_{n\le N}^{log}\max\{ \sup_{g\in G}
Re\left( e(\beta)\mathbb{E}_{h\le H}\mu(n+h)F(g^hx_0)\right),0\}>\tau\}
	    \end{align*}
	    is an infinite set. Moreover, putting $\sigma=\frac{\tau^2}{200}$ and by induction we can find strictly
increasing sequences $\{H_i\}_{i=1}^{\infty}$ of $E$ and  $\{N_i\}_{i=1}^\infty $
of natural numbers  such that for each $i\in \mathbb{N}$ one has	
	    \begin{align}\label{A-00}
	  H_i<\sigma N_i^\sigma<\frac{\sigma}{10}H_{i+1}^\sigma,
		\end{align}
		and there exist $g_{n,i}\in G$ for $ 1\le n\le N_i$ satisfying
\begin{align}\label{getau} \mathbb{E}_{n\le N_i}^{log}\max\{Re\left(e(\beta)\mathbb{E}_{h\le H_i}\mu(n+h)F(g_{n,i}^hx_0)\right),0\}>\tau.
		\end{align}
For $i\in \mathbb{N}$, let $M_i=\sum_{n=1}^{N_i}\frac{1}{n}$
and
\begin{align}\label{A-2}
	S_i=\{n\in [1,N_i]\cap \mathbb{Z}: 	Re\left(e(\beta)\mathbb{E}_{h\le H_i}\mu(n+h)F(g_{n,i}^hx_0)\right)>\frac{\tau}{2}\}.
\end{align}
Then by \eqref{wF} and \eqref{getau} we have
		\begin{align}\label{eq-31-1}
		\sum_{n\in S_i}\frac{1}{n}>\frac{\tau}{2}M_i.
		\end{align} 	
	Notice that $\lim_{N\to+\infty}\frac{\sum_{n\le  N^\sigma}\frac{1}{n}}{\sum_{n\le N}\frac{1}{n}}=\sigma.$
So, when $i\in\mathbb{N}$ large enough we have
		\begin{align*}\sum_{n\in S_i\setminus[1, N_{i}^\sigma]}\frac{1}{n}\overset{\eqref{eq-31-1}}
>\frac{\tau}{2}M_i-\sum_{n\le  N_{i}^\sigma}\frac{1}{n}>\frac{\tau}{2}M_i-2\sigma M_{i}>\frac{\tau}{4}M_i.
		\end{align*}
		Hence we can select $S_i'\subset S_i\setminus[1,N_{i}^\sigma]$ with each gap not less than $2H_i$ and
		\begin{align}\label{A-3}
			\sum_{n\in S'_i}\frac{1}{n}> \frac{\tau M_i}{8H_i}
		\end{align}
		for $i\in\mathbb{N}$ large enough.

Define $y: \mathbb{Z}\to G/\Gamma\cup\{p\}$ such that
		$$y(n+h):=g_{n,i}^hx_0\text{\quad for }n\in S_i', h=1,2,\cdots,H_i,i\in\mathbb{N}$$
		and $y(m)=p$ for $m\in \mathbb{Z}\setminus \bigcup_{i=1}^{\infty} \bigcup_{n\in S_i'}\{n+1,n+2,\cdots, n+H_i\}$.

Clearly $y$ is well defined since $N_{i+1}>N_i+H_i$ by \eqref{A-00}. Then one has by \eqref{A-2} and \eqref{A-3} that
\begin{align*}
Re\left(e(\beta)\sum_{n\in S_i'}\frac{1}{n}\sum_{h\le H_i}\mu(n+h)\widetilde F(\sigma^{n+h}y)\right)>\frac{\tau^2}{16}M_i
\end{align*}
	for $i\in\mathbb{N}$ large enough.
This implies
\begin{align}\label{eq-17-1}
\big|\sum_{n\in S_i'}\frac{1}{n}\sum_{h\le H_i}\mu(n+h)\widetilde F(\sigma^{n+h}y)\big|>\frac{\tau^2}{16}M_i
\end{align}
	for $i\in\mathbb{N}$ large enough.
	 Moreover, for $i\in\mathbb{N}$ large enough,
		\begin{align*}
		&\left|\sum_{n\in S_i'}\frac{1}{n}\sum_{h\le H_i}\mu(n+h)\widetilde F(\sigma^{n+h}y)-\sum_{n\in S_i'}
\sum_{h\le H_i}\frac{\mu(n+h)\widetilde F(\sigma^{n+h}y)}{n+h}\right|\\
 	    \overset{\eqref{wF}}\le &  \sum_{n\in S_i'}	\sum_{h\le   H_i}(\frac{1}{n}-\frac{1}{n+h})\le \sum_{n\in S_i'}
 	\sum_{h\le H_i}\frac{H_i}{n(n+H_i)}\\
	    \le& \sum_{n\in S_i'}	\frac{H_i}{n(N_i^\sigma+H_i)}\le \sum_{n\in S_i'}	\frac{H_i}{nN_i^\sigma}\\
			\overset{\eqref{A-00}}\le& \sigma\sum_{n\in S_i'}\frac{1}{n}\le \frac{\tau^2}{32}M_i.
		\end{align*}
Combining this inequality with \eqref{eq-17-1}, one has
			\begin{align}\label{tau}
		  \left|\sum_{N_{i}^\sigma<n
			\le N_i+H_i}\frac{\mu (n )\widetilde F(\sigma^{n}y)}{n}\right|=\left|\sum_{n\in S_i'}
\sum_{h\le H_i}\frac{\mu(n+h)\widetilde F(\sigma^{n+h}y)}{n+h}\right|\ge \frac{\tau^2M_i}{32}
	\end{align}	
	 for $i\in\mathbb{N}$ large enough. Thus
		\begin{align*}
		&\left|\frac{1}{M_i}\sum_{n\le N_i}\frac{\mu (n )\widetilde F(\sigma^{n}y)}{n}\right|\\
		\ge&\frac{1}{M_i} \left|\sum_{N_{i}^\sigma<n\le N_i+H_i}\frac{\mu (n )\widetilde F(\sigma^{n}y)}{n}\right|
 -\frac{1}{M_i}\sum_{n\le N_{i}^\sigma\text{ or}\atop N_i<n\le N_i+H_i}\left|\frac{\mu (n )\widetilde F(\sigma^{n}y)}{n}\right|\\
		\overset{\eqref{tau}}\ge&\frac{\tau^2}{32}-\frac{\|\widetilde F\|_\infty}{M_i}
\sum_{ N_i<n\le N_i+H_i}\frac{1}{n}-\frac{\|\widetilde F\|_\infty}{M_i}\sum_{ n\le N_{i}^\sigma}\frac{1}{n}\\
		\overset{\eqref{wF}}\ge&\frac{\tau^2}{32}-\frac{H_i}{N_i}-2\sigma\overset{\eqref{A-00}}\ge \frac{\tau^2}{32}-3\sigma
		\overset{\eqref{A-00}}\ge \frac{\tau^2}{100}
		\end{align*}
 for $i\in\mathbb{N}$ large enough. This deduces that
		\begin{align*}\limsup_{N\to\infty}|\mathbb{E}_{n\le N}^{log}
		\mu (n)\widetilde F(\sigma^ny)|\ge \frac{\tau^2}{100}>0.
		\end{align*}
		Therefore, $y$ is the point as required.
		
		\medskip
	{\it \noindent {\bf Step 2.}  $(X_y,\sigma)$ has polynomial mean complexity.}
	
	Recall that $X_y =\overline{\{\sigma^ny:n\in\mathbb{Z}\}}$ is a compact $\sigma$-invariant
subset of $\left(G/\Gamma\cup\{p\}\right)^{\mathbb{Z}}$.
		The metric on $\left(G/\Gamma\cup\{p\}\right)^{\mathbb{Z}}$ is defined by
	\begin{align}\label{metric-1}D(x,x')=\sum_{n\in\mathbb{Z}}\frac{d(x(n),x'(n))}{2^{|n|+2}}\end{align}
	for $x=(x(n))_{n\in \mathbb{Z}}, x'=(x'(n))_{n\in \mathbb{Z}}\in \left(G/\Gamma\cup\{p\}\right)^{\mathbb{Z}}$.
		By Proposition \ref{lem-33}, we can find $k>1$ such that
	\begin{align}\label{limit}
		\lim_{n\to+\infty}\frac{s_{n}(Poly(G/\Gamma),\epsilon)}{n^{k}}=0\text{	for all } \epsilon>0.	
	\end{align}	
	Now we are going to show that $$\liminf_{n\to+\infty}\frac{S_{n}(D,\sigma,X_y,\epsilon)}{n^{k+1}}=0\text{ for all }\epsilon>0.$$

	For $n\in\mathbb{Z}_+$ and $-n\le q\le n$, let  $X_{n,q}$ be the collection of all
points $z\in \left(G/\Gamma\cup\{p\}\right)^{\mathbb{Z}}$ with
	$$z(j)=\begin{cases} p & \text{ if }-n\le j< q \\
	 {\bf g}(j)\Gamma & \text{ if } q\le j\le n
     \end{cases},	
	$$
	  where ${\bf g}$ is some polynomial sequence of $G$  adapted to the lower central
series filtration; and let $X_{n,q}^*$ be the collection of all points $z\in \left(G/\Gamma\cup\{p\}\right)^{\mathbb{Z}}$ with
	  	$$z(j)=\begin{cases} {\bf g}(j)\Gamma  & \text{ if }-n\le j< q \\
	  p & \text{ if } q\le j\le n
	  \end{cases},	
	  $$
	  where ${\bf g}$ is some polynomial sequence of $G$  adapted to the lower central series filtration.

	  For $i\in\mathbb{N}$, put $t_i=[H_i/2]$, where $[u]$ is the integer part of the real number $u$. Then
 \begin{align}\label{11}
	X_y\subset \bigcup\limits_{-t_i\le q\le t_i} X_{t_i,q}\cup\bigcup\limits_{-t_i\le q\le t_i} X_{t_i,q}^*\cup\{\sigma^jy:-H_i\le j\le H_i\}.
	\end{align}
In fact, since $H_{j+1}>N_j+H_j$ for all $j\in \mathbb{N}$, one has
\begin{align}\label{orbit-belong-set}
\sigma^ny\in \bigcup\limits_{-t_i\le q\le t_i} X_{t_i,q}\cup\bigcup\limits_{-t_i\le q\le t_i} X_{t_i,q}^*
\cup\{\sigma^jy:-H_i\le j\le H_i\} \text{ for all } n\in \mathbb{Z}
\end{align}
 by the construction of $y$.  It is not hard to see that $X_{t_i,q},X_{t_i,q}^*$ are all compact subsets of
 $\left(G/\Gamma\cup\{p\}\right)^{\mathbb{Z}}$ for each $-t_i\le q \le t_i$ and $i\in\mathbb{N}$ by
 Lemma \ref{lem-2-5}. Hence the set in right part of \eqref{11} is also a compact subset of
 $\left(G/\Gamma\cup\{p\}\right)^{\mathbb{Z}}$. Now \eqref{11} follows from the fact \eqref{orbit-belong-set}.

 Now we fix $\epsilon>0$. We have the following Claim.

 \medskip
 {\it \noindent{\bf Claim.} For $i\in\mathbb{N}$ large enough, one has	\begin{enumerate}
 		\item[(1)] $S_{[t_i/2]}(D,\sigma,X_{t_i,q},\epsilon)\le s_{t_i}(Poly(G/\Gamma),\frac{\epsilon}{2})$ for all $q\in [-t_i,t_i]\cap\mathbb{Z}$.
 		\item[(2)] $S_{[t_i/2]}(D,\sigma,X^*_{t_i,p},\epsilon)\le s_{t_i}(Poly(G/\Gamma),\frac{\epsilon}{2})$ for all $q\in [-t_i,t_i]\cap\mathbb{Z}$.
 \end{enumerate}
}
  \begin{proof}[Proof of Claim]
  	We prove Claim (1) firstly.
  For $i\in\mathbb{N}$ and $-t_i\le q\le t_i$,  we let $\pi_{i,q}: Poly(G/\Gamma)\to X_{t_i,q}$ be defined by
 $$\pi_{i,q}(z)(j)=\begin{cases} p & \text{ if }-t_i\le j<q\\
 z(j) & \text{ otherwsie}\end{cases}$$
 for $z\in  Poly(G/\Gamma)$. For $i\in\mathbb{N}$ large enough, if $z,\widetilde z\in Poly(G/\Gamma)$
 with $d(z(j),\widetilde z(j))<\frac{\epsilon}{2}$ for all $-t_i\le j\le t_i$, then for $q\in[-t_i,t_i]\cap\mathbb{Z}$
	{\small	\begin{align}\label{dist}\begin{split}
		&\bar{D}_{[t_i/2]}(\pi_{i,q}(z),\pi_{i,q}(\widetilde z))=\frac{1}{[t_i/2]}\sum_{l=0}^{[t_i/2]-1}D(\sigma^lz,\sigma^l\widetilde z)\\	
		\overset{\eqref{metric-1}}=& \frac{1}{[t_i/2]}\sum_{l=0}^{[t_i/2]-1}\sum_{n\in\mathbb{Z}}\frac{d(z(n+l),\widetilde z(n+l))}{2^{|n|+2}}\\
		\le& \frac{1}{[t_i/2]}\sum_{l=0}^{[t_i/2]-1}\left(\sum_{|n|\le [t_i/2] }\frac{d(z(n+l),
\widetilde z(n+l))}{2^{|n|+2}}+\sum_{|n|> [t_i/2] }\frac{d(z(n+l),\widetilde z(n+l))}{2^{|n|+2}}\right)\\
		\le &\frac{1}{[t_i/2]}\sum_{l=0}^{[t_i/2]-1}\left( \frac{\epsilon}{2}+\frac{diam (G/\Gamma)}{2^{t_i/2}}\right)
		<\epsilon,
		\end{split}
		\end{align}}
	where we use the fact $t_i\to+\infty$ as $i\to+\infty$ in the last inequality. Notice that the
map $\pi_{i,q}$ is surjective for all $i\in\mathbb{N}$ and $-t_i\le q\le t_i$. By \eqref{dist}, for $i\in\mathbb{N}$ large enough, one has
	$$S_{[t_i/2]}(D,\sigma,X_{t_i,q},\epsilon)\le s_{t_i}(Poly(G/\Gamma),\frac{\epsilon}{2})\text{ for all }q\in [-t_i,t_i]\cap\mathbb{Z}.$$
	By the similar arguments one has Claim (2). This ends the proof of Claim.
	\end{proof}

Hence by  the above Claim and \eqref{11}, one has
  		\begin{align*}
  S_{[t_i/2]}(D,\sigma,X_y,\epsilon)&\le (2H_i+1)+ \sum_{q=-t_i}^{t_i}\left(S_{[t_i/2]}(D,\sigma,X_{t_i,q},\epsilon)+
  S_{[t_i/2]}(D,\sigma,X_{t_i,q}^*,\epsilon)\right)\\  \nonumber
  &\le (2H_i+1)+(2t_i+1)s_{t_i}(Poly(G/\Gamma),\frac{\epsilon}{2})
  \end{align*}	
for $i\in\mathbb{N}$ large enough.
Combining this with \eqref{limit}	
		\begin{align*}
		\liminf_{n\to+\infty}\frac{S_{n}(D,\sigma,X_y,\epsilon)}{n^{k+1}}
		&\le \liminf_{i\to+\infty}\frac{S_{[t_i/2]}(D,\sigma,X_y,\epsilon)}{[t_i/2]^{k+1}}\\
		&\le \liminf_{i\to+\infty}\frac{ (2H_i+1)+(2t_i+1)s_{t_i}(Poly(G/\Gamma),\frac{\epsilon}{2})}{[t_i/2]^{k+1}}\\
		&=0,
		\end{align*}where we used the assumption $t_i=[H_i/2]$.
This implies that $(X_y,\sigma)$ has polynomial mean complexity, since the above inequality is true for all
$\epsilon>0$. This ends the proof of {\bf Step 2}.
	\end{proof}

\subsection{Proof of (3) implies (2) in Theorem \ref{thm-1}}\label{se-2}

To get the proof we first discuss a t.d.s. with the so-called small boundary property, then
we obtain a key proposition for the proof, and finally we give the proof. We start with
the notion of small boundary property.

	For a t.d.s. $(X,T)$, a subset $E$ of  $X$ is called {\it $T$-small} (or simply small when there is no diffusion) if
		$$\lim_{N\to+\infty}\frac{1}{N}\sum_{n=0}^{N-1}{\bf1}_{E}(T^nx)=0$$
uniformly for $x\in X$. It is not hard to show that a closed subset $E$ of $X$ is small if and only if $\nu(E)=0$ for
all $\nu\in\mathcal{M}(X,T)$.  For a subset $U$ of $X$, we say $U$ has {\it small boundary}
if $\partial U$ is small. We say $(X,T)$ has {\it small boundary property} if for any $x\in X$
and any open neighborhood $V$ of $x$, there exists an open neighborhood $W$ of $x$ such that $W\subset V$ and $W$ has small boundary.
The following lemma indicates that when $X$ has the small boundary property then the logarithmic Sarnak conjecture can be
verified through easier conditions.
	\begin{lem}\label{lem-2.8}
		Let $(X,T)$ be a t.d.s.  with   small boundary
		property. Then the logarithmic Sarnak conjecture holds for $(X,T)$ if and only if for any
subset $U$ of $X$ with small boundary, one has
		\begin{align}\label{llem-1}
		\lim_{N\to+\infty} \mathbb{E}_{n\le N}^{log}{\bf1}_{U}(T^nx)\mu(n)=0
		\end{align}
	for all $x\in X$.
\end{lem}
	\begin{proof} First we assume that \eqref{llem-1} holds for any subset $U$ of $X$ with
small boundary and $x\in X$. For a given  $f\in C(X)$ and fixed $\delta>0$, let $$\epsilon=\epsilon(\delta)
=\sup_{x,y\in X,d(x,y)<\delta}|f(x)-f(y)|.$$ Let $\mathcal{P}=\{P_1,P_2,\cdots,P_k\}$ be a partition of $X$
with diameter small than $\delta$ and  each element of $\mathcal{P}$ has small boundary. For $1\le i\le k$,
we fix points $x_i\in P_i$ and define $\bar f(x)=f(x_i)$ if $x\in P_i$.  Then $\bar f(x)=\sum_{i=1}^k f(x_i){\bf1}_{P_i}(x)$ and by \eqref{llem-1},
		$$\lim_{N\to+\infty}\mathbb{E}_{n\le N}^{log}\bar f(T^nx)\mu(n)=0$$
		for all $x\in X$.
      Since $\|\bar f-f\|_{\infty}\le \epsilon$,  we have
		\begin{align*}
		&\hskip0.5cm \limsup_{N\to+\infty}|\mathbb{E}_{n\le N}^{log} f(T^nx)\mu(n)|\\
		&\le \limsup_{N\to+\infty}|\mathbb{E}_{n\le N}^{log}\bar f(T^nx)\mu(n)|+\limsup_{N\to+\infty}\mathbb{E}_{n\le N}^{log}
		\|\bar f-f\|_{\infty}\cdot|\mu(n)|\\
		&\le \epsilon
		\end{align*}
		for all $x\in X$.
		By taking $\delta\to 0$ and then $\epsilon\to0$, one has $$\lim_{N\to+\infty}\mathbb{E}_{n\le N}^{log}
 f(T^nx)\mu(n)=0$$ for all $x\in X$.	This implies the logarithmic Sarnak conjecture holds for $(X,T)$ since $f$ is arbitrary.
		
		Conversely, we assume that the logarithmic Sarnak conjecture holds for $(X,T)$.
Let $U$ be a subset  of $X$ with small boundary. Fix $\delta>0$. By a result of Shub and Weiss (see \cite[P.537]{SW}),
we can find $\epsilon>0$ such that for $N$ large enough,
		\begin{align*}
		\frac{1}{N}\sum_{n=1}^{N}{\bf1}_{B(\partial U,\epsilon)}(T^nx)\le \frac{\delta}{2},
		\end{align*}
		for all $x\in X$, where $B(\partial U,\epsilon)=\{y\in X: d(y,\partial U)<\epsilon\}$.
Moreover, for $N$ large enough,
         \begin{align}\label{llem-2}
		\begin{split}
        \mathbb{E}_{n\le N}^{log}{\bf1}_{B(\partial U,\epsilon)}(T^nx)&= \frac{1}{M_N}\sum_{n=1}^{N}\frac{{\bf1}_{B(\partial U,\epsilon)}(T^nx)}{n}\\
        &=\frac{1}{M_N}\big(\frac{S_N(x)}{N}+\sum_{j=1}^{N-1}\frac{S_j(x)}{j} \frac{1}{j+1}\big)\\
        &\le \delta
		\end{split}
        \end{align}
		for all $x\in X$, where
 we simply write $M_N=\sum_{n=1}^N\frac{1}{n}$ and $S_j(x)=\sum_{n=1}^j {\bf1}_{B(\partial U,\epsilon)}(T^nx)$ for $j\in \mathbb{N}$.

		Using Urysohn's lemma, there exists a continuous function $h:X\to\mathbb{R}$ with $0\le h\le 1$ such that
$h(x)=1$ for $x\in U\setminus B(\partial U,\epsilon)$ and $h(x)=0$ for $x\in X\setminus \big(U\cup B(\partial U,\epsilon)\big)$.
Since the logarithmic Sarnak conjecture holds for $(X,T)$, one has
		$$\lim_{N\to+\infty} \mathbb{E}_{n\le N}^{log} h(T^nx)\mu(n)=0$$
		for all $x\in X$. Combining this equality with \eqref{llem-2}, we obtain
		\begin{align*}
		&\hskip0.5cm \limsup_{N\to+\infty}|\mathbb{E}_{n\le N}^{log} 1_U(T^nx)\mu (n)|\\
		&\le  	\limsup_{N\to+\infty}|\mathbb{E}_{n\le N}^{log} h(T^nx)\mu (n)|
		+\limsup_{N\to+\infty}\mathbb{E}_{n\le N}^{log} |h(T^nx)-1_U(T^nx)|\\
		&\le \limsup_{N\to+\infty}\mathbb{E}_{n\le N}^{log} 1_{B(\partial U,\epsilon)}(T^nx)\le \delta
		\end{align*}
		for all $x\in X$. By taking $\delta\to 0$, we have 	
		$$\lim_{N\to+\infty}\mathbb{E}_{n\le N}^{log}1_U(T^nx)\mu(n)=0$$
		for all $x\in X$. This ends the proof of Lemma \ref{lem-2.8}.
	\end{proof}

The next lemma concerns the coding of a subset with small boundary.
	\begin{lem}\label{lem-2.9}
		Let $(X,T)$ be a t.d.s. and $U$ be a  subset of $X$ with small boundary. For $x\in X$,
we associate  an $\hat x\in \{0,1\}^{\mathbb{Z}}$ such that $\hat x(n)=1$ if $T^nx\in U$ and $\hat x(n)=0$ otherwise.
Then for $\delta>0$ there exist $\epsilon>0$ and  $N_\delta\in\mathbb{N}$ such that for all
$N\ge N_\delta$ and any $x_1,x_2\in X$ with $\bar d_N(x_1,x_2)< \epsilon$ one has $$\sharp\{0\le n\le N-1:\hat {x_1}(n)
\neq\hat{x_2}(n)\}\le2\delta N.$$
	\end{lem}
	\begin{proof}
	 We 	fix an $\delta\in(0,+\infty)$ and a nonempty subset  $U$  of $X$ with small boundary. By a result of Shub and Weiss
(see \cite[P.537]{SW}) there exist $N_\delta\in \mathbb{N}$ and $\epsilon_0\in (0,+\infty)$ such that
		\begin{align}\label{eq-20-1}\sup_{x\in X,N\ge N_\delta}\frac{1}{N}\sum_{n=0}^{N-1}{\bf 1}_{B(\partial U,\epsilon_0)}(T^nx)<\delta.
		\end{align}
We notice that  $\overline{U\setminus B(\partial U,\epsilon_0)}\cap\overline{X\setminus U}=\emptyset$ and $\overline{(X\setminus U)\setminus B(\partial U,\epsilon_0)}\cap\overline{U}=\emptyset$. Thus we can find $\epsilon\in(0,\delta^2)$ such that when $x,y\in X$ with $d(x,y)<\sqrt{\epsilon}$, if $x\in U\setminus B(\partial U,\epsilon_0)$ (resp. $x\in (X\setminus U)\setminus B(\partial U,\epsilon_0)$), then $y\in U$ (resp. $y\in X\setminus U$).
		 We are to show that $\epsilon$ is the constant as required.  We fix $N\ge N_\delta$ and $x_1,x_2\in X$ with $\bar d_N(x_1,x_2)< \epsilon$.
		Set $$C=\{0\le n\le N-1:T^nx_1\in B(\partial U,\epsilon_0)\}.$$
		By \eqref{eq-20-1},  $\sharp C\le \delta N$. Put
		$$\mathcal{A}=\{0\le n\le N-1: d(T^nx_1,T^nx_2)<\sqrt{\epsilon} \}.$$
One has $\sharp \mathcal{A}\ge(1-\sqrt{\epsilon})N$ and $\hat {x_1}(n)=\hat{x_2}(n) \text{ for all }n\in \mathcal{A}\setminus C$.
Therefore,
		$$\sharp\{0\le n\le N-1:\hat {x_1}(n)=\hat{x_2}(n)\}\ge \sharp \mathcal{A}-\sharp C\ge(1-\sqrt{\epsilon}-\delta)N.$$
		Since $\delta>\sqrt{\epsilon}$, one has
		$$\sharp\{0\le n\le N-1:\hat {x_1}(n)\neq\hat{x_2}(n)\}\le2\delta N.$$
		This ends the proof of Lemma \ref{lem-2.9}.
	\end{proof}
Recall that the metric on $\{0,1\}^\mathbb{Z}$ is defined by
\begin{align}\label{metric}d(x,y)=\sum_{n\in\mathbb{Z}}\frac{|x(n)-y(n)|}{2^{|n|+2}}.\end{align}
for $x=(x(n))_{n\in \mathbb{Z}}, y=(y(n))_{n\in \mathbb{Z}}\in \{0,1\}^{\mathbb{Z}}$.
We have the following lemma which is key for the proof of (3) implies (2) in Theorem \ref{thm-1}.

Now we show a key proposition for the proof of (3) implies (2) in Theorem \ref{thm-1}.

\begin{prop}\label{lem-2.10}
Let $(X,T)$ be a t.d.s. 
and $U$ be a  subset of $X$
with small boundary. For $x\in X$, we associate an $\hat x\in \{0,1\}^{\mathbb{Z}}$ such that $\hat x(n)=1$
if $T^nx\in U$ and $0$  if $T^nx\in X\setminus U$. Then for each $\delta>0$ we can find $\epsilon:=\epsilon(\delta)>0$ such that
		$S_N(d,\sigma,\hat{X},\delta)\le  S_N(d, T,X,\epsilon)$
		for $N\in \mathbb{N}$ large enough,
where $\hat{X}=\overline{\{\hat x:x\in X\}}$ and $\sigma:\{0,1\}^\mathbb{Z}\rightarrow \{0,1\}^{\mathbb{Z}}$
 is the left shift.
	\end{prop}
	\begin{proof}

	 We fix an $\delta>0$  and a nonempty subset  $U$  of $X$ with small boundary. We are to find $\epsilon\in(0,+\infty)$ such that $S_N(d,\sigma,\hat{X},\delta)\le  S_N(d, T,X,\epsilon)$ for $N$ large enough. To do this, we choose $L\in\mathbb{N}$ and $\delta'>0$ such that 
		\begin{align}\label{78-1}
		4\delta'L+\frac{2}{2^L}<\delta.
		\end{align}
		By Lemma \ref{lem-2.9},  there exists $\epsilon:=\epsilon(\delta')>0$ such that for $N\in\mathbb{N}$
large enough and $x_1,x_2\in X$ with $\bar d_N(x_1,x_2)< \epsilon$ one has
		\begin{align}\label{delta}
		\sharp\{0\le n\le N-1:\hat {x_1}(n)\neq\hat{x_2}(n)\}\le2\delta' N.
		\end{align}
		Fix $x_1,x_2\in X$ with $\bar d_N(x_1,x_2)< \epsilon$ and put {\small\begin{align*}
		\mathcal{C}_N=\{0\le n\le N-1:\hat {x_1}(n+l)\neq\hat{x_2}(n+l)\text{ for some }-L+1\le l\le L-1\}.
		\end{align*}}
	By \eqref{delta}, we have  for $N\in\mathbb{N}$ large enough
	 $$\sharp\mathcal{C}_N\le 4\delta'LN.$$
Notice that $d(\sigma^n\hat{x_1},\sigma^n\hat{x_2})\le 1$ for $n\in\mathcal{C}_N$. One has
		\begin{align*}
		\bar d_N(\hat{x_1},\hat{x_2})=&\frac{1}{N}\left(\sum_{n\in \mathcal{C}_N}d(\sigma^n\hat{x_1},\sigma^n\hat{x_2})+\sum_{n\in[0,N-1]\setminus\mathcal{C}_N}d(\sigma^n\hat{x_1},\sigma^n\hat{x_2})\right) \\
			 \overset{\eqref{metric}}\le& \frac{1}{N}\left(\sum_{n\in \mathcal{C}_N}1+\sum_{n\in[0,N-1]\setminus\mathcal{C}_N}\frac{2}{2^L}\right) \\ =&\frac{1}{N}\left(\sharp\mathcal{C}_N+\frac{2}{2^L}(N-\sharp\mathcal{C}_N)\right)  \\
	\overset{\eqref{78-1}}	\le&4\delta'L+\frac{2}{2^L}<\delta.
		\end{align*}
	Therefore, $S_N(d,\sigma,\hat{X},\delta)\le  S_N(d, T,X,\epsilon)$ for $N\in\mathbb{N}$ large enough and $\epsilon$ is the constant as required.
This ends the proof of Proposition \ref{lem-2.10}.
	\end{proof}

For a t.d.s. $(X,T)$, Lindenstrauss and Weiss \cite{LW} introduced the notion of {\it mean
	dimension}, denoted by $mdim(X,T)$. 
It is well known that for a
	t.d.s. $(X,T)$, if $h_{top}(T )<\infty$ or the topological dimension of $X$ is finite, then $mdim(X,T ) = 0$
(see Definition 2.6 and Theorem 4.2 in \cite{LW}).

Now we are ready to finish the proof of Theorem \ref{thm-1},
	
	\begin{proof}[Proof of Theorem \ref{thm-1}: (3)$\Longrightarrow$(2).]
		Assume that Theorem \ref{thm-1} (3) holds. Now we are going to show that Theorem \ref{thm-1} (2) holds.
Assume the contrary that Theorem \ref{thm-1} (2) doesn't hold, then there exists a t.d.s. $(X,T)$  with polynomial mean complexity
such that the logarithmic Sarnak conjecture does not hold for $(X,T)$.
		
Let $(Y,S)$ be 	an irrational rotation on the circle. Then $(X\times Y,T\times S)$ has polynomial
mean  complexity as well as zero mean dimension
and  admits a non periodic minimal factor
		$(Y,S)$. Hence $(X \times Y,T\times S)$ has small boundary
		property by \cite[Theorem 6.2]{L}. Since the logarithmic Sarnak conjecture
does not hold for $(X,T)$, neither does $(X \times Y,T\times S)$. By Lemma \ref{lem-2.8}, there is a subset
$U$ of $X\times Y$ with small boundary and $w\in X\times Y$ such that
		$$\limsup_{N\to+\infty}\mathbb{E}_{n\le N}^{log} {\bf1}_{U}\big((T\times S)^nw\big)\mu(n)>0.$$
		Combining this with Proposition \ref{lem-2.10}, the $\{0,1\}$-symbolic
system $(\overline{\{\hat z:z\in X\times Y\}},\sigma)$ has polynomial mean complexity and
		$$\limsup_{N\to+\infty}\mathbb{E}_{n\le N}^{log} F_0(\sigma^n\hat w)\mu(n)>0,$$
		where $F_0(\hat z)=\hat z(0)$ for $z\in X\times Y$, which contradicts the assumption
that  Theorem \ref{thm-1} (3) holds. This ends the proof of (3)$\Longrightarrow$(2) in Theorem \ref{thm-1},
and hence the proof of Theorem \ref{thm-1}.
	\end{proof}
	

	\section{Proof of Theorem \ref{thm-2}}
	In this section we will prove Theorem \ref{thm-2}.  Firstly, we  recall the definition of
packing dimension. Let $X$ be a metric space endowed with a metric $d$ and $E$ be a subset of $X$.
We say that a collection of balls $\{U_n\}_{n\in\mathbb{N}}\subset X$ is a {\it $\delta$-packing} of $E$ if the diameter of the
balls is not larger than  $\delta$, they are pairwise
	disjoint and their centres belong to $E$. For $\alpha\in \mathbb{R}$, the {\it $\alpha$-dimensional pre-packing
	measure of $E$} is given by
	$$P(E,\alpha)=\lim_{\delta\to0}\sup\{\sum_{n\in\mathbb{N}}
	diam(U_n)^\alpha\},$$
	where the supremum is taken over all $\delta$-packings of $E$. The {\it $\alpha$-dimensional packing
	measure of $E$} is defined by
	$$p(E, \alpha) = \inf\{\sum_{i\in\mathbb{N}}
	P(E_i
	, \alpha)\},$$
	where the infimum is taken over all covers $\{E_i\}_{i\in \mathbb{N}}$ of $E$. Finally, we define the {\it packing
	dimension of $E$} by
	$$Dim_P E=\sup\{\alpha:p(E, \alpha) = +\infty\} = \inf\{\alpha:p(E, \alpha) = 0\}.$$

For $x\in [0,1]$ and $r>0$, let $B(x,r)=\{y\in[0,1],|x-y|<r\}$.	To prove Theorem \ref{thm-2}, we need
several lemmas. We begin with the following lemma (see \cite{FLR}).
	\begin{lem}\label{Packing}
		Let $\mu$ be a Borel probability measure on $[0,1]$. Then $$Dim^*\mu=\inf\{Dim_PE: E\subset [0,1]\text{ with } \mu(E^c)=0\},$$
		where $Dim^*\mu=ess\sup \limsup_{r\to 0}\frac{\log\mu(B(x,r))}{\log r}$.
	\end{lem}
We also need the following lemma \cite[Theorem 2.1]{M95}.	
	\begin{lem}
		\label{lem-2.1}
		
		Let ${\mathcal
			B}=\{B(x_i,r_i)\}_{i\in \mathcal I}$ be a family of  open
		balls in $[0,1]$. Then there exists a finite or countable subfamily
		${\mathcal B'}=\{B(x_i,r_i)\}_{i\in {\mathcal I}'}$ of pairwise
		disjoint balls in ${\mathcal B}$ such that
		$$\bigcup_{B\in {\mathcal B}} B\subseteq \bigcup_{i\in {\mathcal I}'}B(x_i,5r_i).$$
	\end{lem}
	
	Let $\mathbb{T}$ be the unit circle on the complex plane $\mathbb{C}$. Recall that $e(t)=e^{2\pi t}$ for any $t\in \mathbb{R}$.
We will prove the following lemma  by using Lemma \ref{Packing} and Lemma \ref{lem-2.1}. Define a metric
$d$ on $[0,1]\times \mathbb{T}$ such that $d((x_1,z_1),(x_2,z_2))=\max\{ |x_1-x_2|,|z_1-z_2|\}$
		for $(x_1,z_1),(x_2,z_2)\in [0,1]\times \mathbb{T}$.
	\begin{lem}\label{LEMMA}Let $C$ be a compact subset of $[0,1]$ with $Dim_P C<\tau$ for some given $\tau>0$. Then the t.d.s.
$T: C\times \mathbb{T}\to  C\times \mathbb{T}$ defined by $T\big(x,e(y)\big)= \big(x,e(y+x)\big)$ satisfies
for any $\rho\in \mathcal{M}(C\times \mathbb{T},T)$ and any $\epsilon>0$,
$$\liminf_{n\rightarrow +\infty} \frac{S_n(d,T,\rho,\epsilon)}{n^\tau}=0.$$
	\end{lem}
	\begin{proof} Fix a constant $\tau_0$ with $Dim_P C<\tau_0<\tau$. For a given  $\rho\in \mathcal{M}(C\times \mathbb{T},T)$
let $m$ be the projection of $\rho$ onto the first coordinate. Fix $\epsilon\in (0,1)$. To prove Lemma \ref{LEMMA}, it suffices to demonstrate
		$$\liminf_{n\rightarrow +\infty} \frac{S_n(d,T,\rho,\epsilon)}{n^\tau}=0.$$

		First we note that $m(C)=1$. Using Lemma \ref{Packing}, one has $Dim^*m<\tau_0$  and there exist a
subset $\widetilde C$ of $C$ and a  constant $r_\epsilon\in (0,1)$ such that
		\begin{itemize}
			\item[(1).] $\widetilde C$ is compact and $m(\widetilde C)>1-\epsilon$;
			\item[(2).] $m(B(x,r))> r^{\tau_0}$ for $0<r\le r_\epsilon$ and $x\in \widetilde C$.
		\end{itemize}

		For any given integer $n>\frac{\epsilon}{10r_\epsilon}$, set $\mathcal{B}_n=\{B(x,\frac{\epsilon}{10n})\}_{x\in \widetilde C}$. By Lemma \ref{lem-2.1},
there exist pairwise disjoint balls $\mathcal{B}'_n=\{B(x_i,\frac{\epsilon}{10n})\}_{i\in\mathcal{I}_n}$ in $\mathcal{B}$ such that
		$$\widetilde C\subset \bigcup_{i\in\mathcal{I}'_n}B(x_i,\frac{\epsilon}{2n}).$$
		Since $\frac{\epsilon}{10n}<r_\epsilon$, one deduces that
		$$m(B(x,\frac{\epsilon}{10n}))>\left(\frac{\epsilon}{10n}\right)^{\tau_0}\text{ for all }x\in \widetilde C.$$
		Therefore, $\mathcal{I}_n$ is finite  since elements in $\mathcal{B}_n'$ are pairwise disjoint. Precisely,
		$$\sharp\mathcal{I}_n\le \left(\frac{10n}{\epsilon}\right)^{\tau_0}.$$ Now we
		put $$E_\epsilon=\{\big(x_i,e(\frac{\epsilon j}{4\pi})\big):i\in\mathcal{I}_n\text{ and } j\in \{0,1,\cdots,[\frac{4\pi}{\epsilon}]\},$$
		where $[\frac{4\pi}{\epsilon}]$ is the integer part of $\frac{4\pi}{\epsilon}$.
		Then, for $n>\frac{\epsilon}{10r_\epsilon}$, it is not hard to verify that
		$$B_{\bar d_n}\Big(\big(x_i,e(\frac{\epsilon j}{4\pi})\big),\epsilon\Big)
		\supset B(x_i,\frac{\epsilon}{2n})\times \{e(t):|t-\frac{\epsilon j}{4\pi}|<\frac{\epsilon}{4\pi}\}$$
		for $i\in\mathcal{I}_n\text{ and } j\in \{0,1,\cdots,[\frac{4\pi}{\epsilon}]\}$. This implies that for $n>\frac{\epsilon}{10r_\epsilon}$ one has
		\begin{align*}
		\rho(\bigcup_{y\in E_\epsilon}B_{\bar d_n}(y,\epsilon))&\ge \rho(\bigcup_{i\in\mathcal{I}_n}B(x_i,\frac{\epsilon}{2n})\times\mathbb{T})=m(\bigcup_{i\in\mathcal{I}_n}B(x_i,\frac{\epsilon}{2n}))\\
		&\ge m( \widetilde C)\ge 1-\epsilon,
		\end{align*}
		and
		$$S_n(d,T,\rho,\epsilon)\le \sharp E_\epsilon\le \sharp \mathcal{I}_n
\times \frac{4\pi}{\epsilon }\le \left(\frac{10n}{\epsilon}\right)^{\tau_0}\times \frac{4\pi}{\epsilon }.$$
		By the fact $\tau_0<\tau$, one has $$\liminf_{n\rightarrow +\infty} \frac{S_n(d,T,\rho,\epsilon)}{n^\tau}=0.$$ This ends the proof of Lemma \ref{LEMMA}.
	\end{proof}

  Now let $p=(0,0)$ be the origin of $\mathbb{C}$. For a sequence $y\in \left(\mathbb{T}\cup\{ p\}\right)^\mathbb{Z}$, let
 {\small $$Gen(y)=\{\mu\in\mathcal{M}(\left(\mathbb{T}
 \cup\{p\}\right)^\mathbb{Z},\sigma):\frac{1}{N_i-M_i}\sum_{M_i<n\le N_i}\delta_{\sigma^ny}\to\mu\text{ for }N_i-M_i\to+\infty\},$$}
 where $\sigma: \left(\mathbb{T}\cup\{p\}\right)^\mathbb{Z}\rightarrow \left(\mathbb{T}\cup\{p\}\right)^\mathbb{Z}$
 is the left shift. Put $X_y=\overline{\{\sigma^ny:n\in \mathbb{Z}\}}$. Then $(X_y,\sigma)$ is a subsystem of
 $(\left(\mathbb{T}\cup\{p\}\right)^\mathbb{Z},\sigma)$.
 It is not hard to see that for $\mu\in Gen(y)$, $\mu(X_y)=1$, and thus
 we can identify $Gen(y)$ with $\mathcal{M}(X_y,\sigma)$. We have

 \begin{lem}\label{lem}Let $C$ be a non-empty compact subset of $[0,1]$ and
 $y\in \left(\mathbb{T}\cup\{p\}\right)^\mathbb{Z}$. Assume that the pair $(y,C)$ meets the
 following

\noindent Property $(*)$: there exist $\{m_1<n_1<m_2<n_2\cdots\}\subset \mathbb{Z}$,
 $\{\theta_k\}_{k\ge 1}\subset C$ and $\{\phi_k\}_{k\ge 1}\subset [0,1]$ such that
 	\begin{enumerate}
 		\item $\lim_{i\to\infty}n_i-m_i=+\infty;$
 		\item $y(j)=p$ for $j\in\mathbb{Z}\setminus\cup_{i\in\mathbb{N}}[m_{i},n_{i});$
 		\item $y(m_i+j)=e(\phi_i+j\theta_i)$ for all $i\ge 1$ and $0\le j< n_{i}-m_i$.
 	\end{enumerate}
 	Then any element in $Gen(y)$ supports on the compact subset
 $$\widetilde{C}=\{(ze(i\theta))_{i\in \mathbb{Z}} \in \mathbb{T}^{\mathbb{Z}}:
 \theta\in C, z\in\mathbb{T}\}\cup\{p\}^\mathbb{Z}.$$
 	\end{lem}
 \begin{proof} Assume that $(y,C)$ meets Property $(*)$ and set
 	$$Z=\{z\in \left(\mathbb{T}\cup\{p\}\right)^\mathbb{Z}:z(-1)=p,z(0)\in\mathbb{T}\}.$$
 	It is clear that  $X_y\setminus \bigcup_{n\in\mathbb{Z}}\sigma^nZ\subset \widetilde{C}$.
 To prove the lemma, it is enough to show that $\mu(\widetilde C)=1$  for all $\mu\in Gen(y)$.
 Since $Gen(y)=\mathcal{M}(X_y,\sigma)$, it is enough to show that $\mu(Z)=0$  for all $\mu\in Gen(y)$.

 Now we fix a $\mu\in Gen(y)$.
 	Then there exist $M_1<N_1, M_2<N_2,\cdots$ such that $\lim_{i\rightarrow +\infty}
 N_i-M_i=+\infty$ and $$\lim_{i\rightarrow +\infty}\frac{1}{N_i-M_i}\sum_{M_i<n\le N_i}\delta_{\sigma^ny}=\mu.$$
 	Since $Z$ is an open subset of $\left(\mathbb{T}\cup\{p\}\right)^\mathbb{Z}$, we have
 	\begin{align*}
 		\mu(Z)&\le \liminf_{i\rightarrow +\infty}\frac{1}{N_i-M_i}\sum_{M_i<n\le N_i}\delta_{\sigma^ny}(Z)\\
 		&=\liminf_{i\to +\infty}\frac{\sharp\{M_i<n\le N_i:\sigma^ny\in Z\}}{N_i-M_i}\\
 		&=\liminf_{i\to+\infty}\frac{\sharp\{M_i<n\le N_i:y(n-1)=p,y(n)\in\mathbb{T}\}}{N_i-M_i}\\
 		&=\liminf_{i\to+\infty}\frac{\sharp\{j\in\mathbb{N}: M_i<m_j\le N_i\}}{N_i-M_i} =0,
 	\end{align*}
 	where the last equality follows from Property $(*)$ (1). This ends the proof of Lemma \ref{lem}.
 \end{proof}	

The next lemma follows easily from the previous ones.

	\begin{lem}\label{lem-1}Assume that $ C$ is a nonempty compact subset of $[0,1]$ with $Dim_P C<\tau$
and $y\in \left(\mathbb{T}\cup\{p\}\right)^\mathbb{Z}$. If $(y,C)$ meets  Property $(*)$ as in Lemma \ref{lem},
then the t.d.s. $(X_y,\sigma)$ satisfies $$\liminf_{n\rightarrow +\infty} \frac{S_n(d,T,\rho,\epsilon)}{n^\tau}=0$$
		for all  $\epsilon>0$ and $\rho\in \mathcal{M}(X_y,\sigma)$.
	\end{lem}
	\begin{proof} Fix a pair $(y,C)$ which meets Property $(*)$  as in Lemma \ref{lem}. Then all measures in $Gen(y)$
support on a compact set $$\widetilde{C}=\{(ze(i\theta))_{i\in \mathbb{Z}}
\in \mathbb{T}^{\mathbb{Z}}: \theta\in C, z\in\mathbb{T}\}\cup\{p\}^\mathbb{Z}.$$

	It is clear that $\widetilde{C}$ is a $\sigma$-invariant compact subset of
	$\left(\mathbb{T}\cup\{p\}\right)^\mathbb{Z}$, that is $(\widetilde C,\sigma)$ is
a t.d.s. Notice that $(\widetilde C,\sigma)$ is a factor of
$(C\times \mathbb{T}\cup\{p\},T)$, where $T:  C\times \mathbb{T}\cup\{p\} \to  C\times \mathbb{T}\cup\{p\}$
with $T(p)=p$ and $T\big(x,e(y)\big)= \big(x,e(y+x)\big)$ for $(x,e(y))\in C\times \mathbb{T}$.
The lemma is immediately from Lemma \ref{LEMMA}.
	\end{proof}

The final lemma we need is the following one.
	
	\begin{lem}\label{3-7} If there exist a non-empty compact subset $C$ of $[0,1]$ and
$\beta\in\mathbb{R}$ such that \begin{align}\label{A-0000}\limsup_{H\to+\infty}\limsup_{N\to+\infty}
\mathbb{E}_{n\le N}^{log}\max\{Re\left(\sup_{\alpha\in C}e(\beta)\mathbb{E}_{h\le H}\mu(n+h)e(h\alpha)\right),0\}>0,
	\end{align}
	then there is $y\in \left(\mathbb{T}\cup\{p\}\right)^{\mathbb{Z}}$ such that $(y,C)$ meets Property $(*)$ as in Lemma \ref{lem} and
	\begin{align}\label{eq-3}\limsup_{N\to\infty}|\mathbb{E}_{n\le N}^{log}
	\mu (n)\widetilde F(\sigma^ny)|>0,	\end{align}
	where $\widetilde F: \left(\mathbb{T}\cup\{p\}\right)^{\mathbb{Z}}\to \mathbb{C}$ is the continuous
function defined by $\widetilde F( z)=z(0)$ if $z(0)\in \mathbb{T}$ and $0$ if $z(0)=p$.
\end{lem}
\begin{proof}
	By the assumption \eqref{A-0000} and the similar arguments as in the proof of Theorem \ref{thm-1}
(2)$\Longrightarrow$(1), we can find $\tau\in(0,1)$, strictly increasing sequences
$\{H_i\}_{i\in\mathbb{N}},\{N_i\}_{i\in\mathbb{N}}$ of natural numbers, series
$\{\alpha_{i,j}\}_{j=1}^{N_i}\subset \mathbb{R}$ , $i=1,2,3\cdots$ and $\beta\in
\{0,\frac{1}{4},\frac{2}{4},\frac{3}{4}\}$ such that for each $i\in \mathbb{N}$ one has	
\begin{align}\label{C2}
 H_i<\sigma N_i^\sigma<\frac{\sigma}{10}H_{i+1}^\sigma\text{ where } \sigma=\frac{\tau^2}{200}
\end{align}
and
\begin{align}\label{C4}	\mathbb{E}_{n\le N_i}^{log}\max\{Re\left(e(\beta)
\mathbb{E}_{h\le H_i}\mu (n + h)e(h\alpha_{n,i})\right),0\}>\tau.
\end{align}
For $i\in \mathbb{N}$, let $M_i=\sum_{n=1}^{N_i}\frac{1}{n}$
and
\begin{align}\label{C23}
	S_i=\{n\in[1,N_i]:Re\left(e(\beta)	\mathbb{E}_{h\le H_i}\mu (n + h)e(h\alpha_{n,i})\right)>\frac{\tau}{2}\}.\end{align}
Then by \eqref{C4}, one has
	\begin{align}\label{C33}
		\sum_{n\in S_i}\frac{1}{n}>\frac{\tau}{2}M_i,
	\end{align}
	Notice that $\lim_{N\to+\infty}\frac{\sum_{n\le  N^\sigma}\frac{1}{n}}{\sum_{n\le N}\frac{1}{n}}=\sigma.$ We have
$$\sum_{n\in S_i\setminus[1,N_{i}^\sigma]}\frac{1}{n}\overset{\eqref{C33}}>\frac{\tau}{2}M_i-\sum_{n\le N_{i}^\sigma}\frac{1}{n}>\frac{\tau}{2}M_i-2\sigma M_{i-1}\overset{\eqref{C2}}>\frac{\tau}{4}M_i
$$	
for $i\in\mathbb{N}$ large enough.
Then we can choose $S_i'\subset S_i\setminus[1,N_{i}^\sigma]$ such that  each gap in $S_i'$ is not less than $2H_i$ and
	\begin{align}\label{C3}\sum_{n\in S'_i}\frac{1}{n}> \frac{\tau M_i}{8H_i}\end{align}
	for $i\in\mathbb{N}$ large enough.
	 Define $y: \mathbb{Z}\to \mathbb{T}\cup\{p\}$ such that
	$$y(j)=e\big((j-n)\alpha_{n,i}\big) \text{ if }j\in [n+1,n+H_i] \text{ for some }i\ge 1\text{ and }n\in S_i',$$
	and $y(j)=p$ for other $j$, where $p$ is the zero of $\mathbb{C}$. It is not hard to see that $y$ is well defined and  meets Property $(*)$.
	
	Now we are going to show that \eqref{eq-3} holds. 	Combining \eqref{C23} with \eqref{C3}, one has
	\begin{align}\label{22-1}
	Re\left(e(\beta)\sum_{n\in S_i'}\sum_{h\le H_i}\frac{\mu (n + h) \widetilde F(\sigma^{n+h}y) }{n}\right)>\frac{\tau}{2}\times H_i\times\sum_{n\in S_i'}\frac{1}{n}>\frac{\tau^2}{16}M_i
	\end{align}
for $i\in\mathbb{N}$ large enough.	Then
\begin{align*}&\left|\sum_{n\in S_i'}\sum_{h\le H_i}\frac{\mu(n+h)\widetilde F(\sigma^{n+h}y)}{n}-\sum_{n\in S_i'}\sum_{h\le H_i}\frac{\mu(n+h)\widetilde F(\sigma^{n+h}y)}{n+h}\right|\\
\le &  \sum_{n\in S_i'}	\sum_{h\le   H_i}(\frac{1}{n}-\frac{1}{n+h})\le \sum_{n\in S_i'}	\sum_{h\le H_i}\frac{H_i}{n(n+H_i)}\\
	\le& \sum_{n\in S_i'}	\frac{H_i}{nN_i^\sigma}
	\overset{\eqref{C2}}\le \sigma\sum_{n\in S_i'}\frac{1}{n}\overset{\eqref{C2}}\le \frac{\tau^2}{32}M_i
\end{align*}
for $i\in\mathbb{N}$ large enough.	
Combining this inequality with \eqref{22-1}, one has

\begin{align*}\left|\sum_{N_{i}^\sigma<n\le N_i+H_i}\frac{\mu (n )\widetilde F(\sigma^{n}y)}{n}\right|&=\left|\sum_{n\in S_i'}\sum_{h\le H_i}\frac{\mu(n+h)\widetilde F(\sigma^{n+h}y)}{n+h}\right|\\
	&\ge Re\left(e(\beta)\sum_{n\in S_i'}\sum_{h\le H_i}\frac{\mu (n + h)y(n+h)}{n}\right)-\frac{\tau^2M_i}{32}\\
&\ge \frac{\tau^2M_i}{32}
\end{align*}
for $i\in\mathbb{N}$ large enough.
Thus
{\small
	\begin{align*}\left|\mathbb{E}_{n\le N_i}^{log}
	\mu (n)\widetilde F(\sigma^ny)\right|&\ge\left|\frac{1}{M_i}\sum_{n\le N_i+H_i}\frac{\mu (n )
\widetilde F(\sigma^ny)}{n}\right|-\frac{1}{M_i}\sum_{N_i<n\le N_i+H_i\text{ or }\atop n
\le N_{i}^\sigma}\left|\frac{\mu (n )\widetilde F(\sigma^ny)}{n}\right|\\
	&\ge \frac{\tau^2}{32}-2\sigma-\frac{H_i}{N_i}\overset{\eqref{C2}}\ge\frac{\tau^2}{100}>0
	\end{align*}}
for $i\in\mathbb{N}$ large enough.
		Therefore, $y$ is the point as required.  This ends the proof of Lemma \ref{3-7}.
\end{proof}

Now we are ready to prove Theorem \ref{thm-2}.
	\begin{proof}[Proof of Theorem \ref{thm-2}] 	
	Assume that  Theorem \ref{thm-2} is not valid. Then there exists a non-empty compact subset
$C$ of $[0,1]$ with $Dim_P C<1$ such that
		$$\limsup_{H\to+\infty}\limsup_{N\to+\infty}\mathbb{E}_{n\le N}^{log}\sup_{\alpha\in C}|
\mathbb{E}_{h\le H}\mu(n+h)e(h\alpha)|>0.$$
		Thus we can find
		 $\beta\in\{0,\frac{1}{4},\frac{2}{4},\frac{3}{4}\}$ such that
		\begin{align*}\limsup_{H\to+\infty}\limsup_{N\to+\infty}\mathbb{E}_{n\le N}^{log}
\max\{\sup_{\alpha\in C}Re\left(e(\beta)\mathbb{E}_{h\le H}\mu(n+h)e(h\alpha)\right),0\}>0.\end{align*}
		By Lemma \ref{3-7}, there is $y\in \left(\mathbb{T}\cup\{p\}\right)^{\mathbb{Z}}$
such that $(y,C)$ meets  Property $(*)$ as in Lemma \ref{lem} and
		\begin{align}\label{eq-333}\limsup_{N\to\infty}|\mathbb{E}_{n\le N}^{log}
		\mu (n)\widetilde F(\sigma^ny)|>0,	\end{align}
		where $\widetilde F: X_y\to \mathbb{R}$ is a continuous function defined by $\widetilde
F(z)=z(0)$ if $z(0)\in \mathbb{T}$ and $0$ if $z(0)=p$. Then, by Lemma \ref{lem-1} and
assumption $Dim_P C<1$, the t.d.s. $(X_y,\sigma)$ satisfies $$\liminf_{n\rightarrow +\infty}
\frac{S_n(d,\sigma,\rho,\epsilon)}{n}=0 \text{	for any } \epsilon>0\text{ and }\rho\in \mathcal{M}(X_y,\sigma).$$
		By Theorem \ref{rem-linear},
		$$\lim_{N\to\infty}\mathbb{E}_{n\le N}^{log}
		\mu (n)\widetilde F(\sigma^ny)=0.$$ This conflicts with \eqref{eq-333} and the theorem
follows. We end the proof of Theorem \ref{thm-2}.
	\end{proof}

\appendix
\section{Proof of Theorem \ref{rem-linear} }\label{Appendix-A}
In this appendix we prove Theorem \ref{rem-linear} following the arguments of the proof of Theorem 1.1' in \cite{HWY}.

Let $(X,T)$ be a t.d.s. with a metric $d$ and sub-linear mean measure complexity.
To prove that the logarithmic Sarnak conjecture holds for $(X,T)$, it is sufficient to show
\begin{align}\label{mineq-1}
	\limsup_{i\rightarrow +\infty}\left|\frac{1}{\sum_{n=1}^{N_i} \frac{1}{n} } \sum_{n=1}^{N_i} \frac{\mu(n)f(T^nx)}{n}\right|<7\epsilon
\end{align}
for any $\epsilon\in (0,1)$ and $f\in C(X)$ with $\max_{z\in X}|f(z)|\le 1$,
$x\in X$ and $\{ N_1<N_2<N_3<\cdots\} \subseteq \mathbb{N}$ such that the sequence
$\mathbb{E}^{log}_{n\le N_i}\delta_{T^nx}$ weakly$^*$ converges to a Borel probability measure $\rho$.

To this aim we will find $L\in \mathbb{N}$, $\{x_1,x_2,\cdots,x_m\}\subset X$ and
$j_n\in \{1,2,\cdots,m\}$ for $n=1,2,3,\cdots$ such that that for large $i$
\begin{align}\label{control-es-1}
	\left|\frac{1}{M_i}\sum_{n=1}^{N_i} \frac{\mu(n)f(T^nx)}{n}-\frac{1}{M_i}\sum_{n=1}^{N_i}
 \Big( \frac{1}{L} \sum_{\ell=0}^{L-1} \frac{\mu(n+\ell) f(T^\ell x_{j_n})}{n}\Big)\right|<5\epsilon
\end{align}
and
\begin{align} \label{control-es-2}
	\left|\frac{1}{M_i}\sum_{n=1}^{N_i} \Big( \frac{1}{L}\sum_{\ell=0}^{L-1} \frac{\mu(n+\ell)
f(T^\ell x_{j_n})}{n}\Big)\right|<2\epsilon.
\end{align}

It is clear that \eqref{mineq-1} follows by \eqref{control-es-1} and \eqref{control-es-2}.
(\ref{control-es-1}) and (\ref{control-es-2}) will be proved in Lemma \ref{A-lem-2} and Lemma \ref{A-lem-3})
respectively, where we write $M_i= \sum_{n=1}^{N_i}
\frac{1}{n}$ for $i\in \mathbb{N}$.

 To prove the two lemmas we firstly choose $\epsilon_1>0$ such that $\epsilon_1<\epsilon^2$ and
\begin{equation}\label{2017-3}
|f(y)-f(z)|<\epsilon\  \text{when}\ d(y,z)<\sqrt{\epsilon_1}.
\end{equation}
Since $\mathbb{E}^{log}_{n\le N_i}\delta_{T^nx}$ weakly$^*$ converges to $\rho$,
it is not hard to verify $\rho\in \mathcal{M}(X,T)$. So, the measure complexity of $(X,d,T,\rho)$
is sub-linear by the assumption of the theorem, and thus there exists $L>0$ such that
\begin{equation}\label{2017-2}
m=S_L(d,T,\rho,\epsilon_1)<\epsilon L.
\end{equation}
This means that there exist $x_1,x_2,\cdots,x_m\in X$ such that
$$\rho\big(\bigcup_{i=1}^m B_{\overline{d}_L}(x_i,\epsilon_1)\big)>1-\epsilon_1>1-\epsilon^2.$$
Put $U=\bigcup_{i=1}^m B_{\overline{d}_L}(x_i,\epsilon_1)$ and $E=\{n\in \mathbb{N}:T^n x\in U\}$.
Then $U$ is open and so
\begin{align}\label{ineq-0}
	\liminf_{i\rightarrow +\infty}\frac{1}{M_i}\sum_{n\in E\cap [1,N_i]}\frac{1}{n}
=\liminf_{i\rightarrow +\infty}\frac{1}{M_i}\sum_{n=1}^{N_i} \frac{\delta_{T^nx}(U)}{n}\ge \rho(U)>1-\epsilon_1.
\end{align}

For $n\in E$, we choose $j_n\in \{1,2,\cdots,m\}$ such that $T^nx\in B_{\overline{d}_L}(x_{j_n},\epsilon_1)$. Hence for $n\in E$, we have
$\overline{d}_L(T^nx,x_{j_n})<\epsilon_1$, i.e.
$$\frac{1}{L}\sum_{\ell=0}^{L-1} d\big( T^\ell(T^nx),T^\ell(x_{j_n})\big)<\epsilon_1$$
and so we have
\begin{equation}\label{2017-4}
\#\{\ell\in [0,L-1]:d(T^{\ell} (T^n x),T^{\ell}x_{j_n})\ge \sqrt{\epsilon_1}\}<L\sqrt{\epsilon_1}<L\epsilon.
\end{equation}
Thus for $n\in E$
\begin{align}\label{ineq-33}
	&\hskip0.5cm \frac{1}{L} \sum_{\ell=0}^{L-1} |f(T^{\ell} (T^n x) )-f(T^\ell x_{j_n})|\nonumber\\
	&\le \frac{1}{L} \big(\epsilon\#\{\ell\in [0,L-1]:d(T^{\ell} (T^n x),T^{\ell}x_{j_n})<\sqrt{\epsilon_1}\}\\
	&\hskip1cm+2\#\{\ell\in [0,L-1]:d(T^{\ell} (T^n x),T^{\ell}x_{j_n})\ge \sqrt{\epsilon_1}\}\big)\nonumber \\
	&<3\epsilon \nonumber,
\end{align}
by using the inequality \ref{2017-3}, (\ref{2017-4}) and the assumption $\max_{x\in X}|f(x)|\le 1$.

For each $n\notin E$, we simply set $j_n=1$.

We first establish  
Lemma \ref{A-lem-2}.
\begin{lem}\label{A-lem-2} For all sufficiently large $i$,
	$$\left|\frac{1}{M_i}\sum_{n=1}^{N_i} \frac{\mu(n)f(T^nx)}{n}-\frac{1}{M_i}\sum_{n=1}^{N_i}
\frac{1}{L} \sum_{\ell=0}^{L-1}\frac{ \mu(n+\ell) f(T^\ell x_{j_n})}{n} \right|<5\epsilon.$$
\end{lem}
\begin{proof} As $\max_{x\in X}|f(x)|\le 1$, it is not hard to see that
	\begin{align}\label{ineq-1}
		\limsup_{i\to+\infty}\left|\frac{1}{M_i}\sum_{n=1}^{N_i} \frac{\mu(n)f(T^nx)}{n}-\frac{1}{M_i}\sum_{n=1}^{N_i}
\frac{1}{L} \sum_{\ell=0}^{L-1} \frac{\mu(n+\ell) f(T^{n+\ell} x)}{n} \right|
		=0.
	\end{align}
	
	By \eqref{ineq-0} once $i$ is large enough,
	\begin{align}\label{ineq-2}
		\frac{1}{M_i}\sum_{n\in E\cap [1,N_i]}\frac{1}{n}>1-\epsilon^2>1-\epsilon.
	\end{align}
	Now
	\begin{align*}
		&\hskip0.5cm \left|\frac{1}{M_i}\sum_{n=1}^{N_i} \frac{1}{L} \sum_{\ell=0}^{L-1}
\frac{\mu(n+\ell) f(T^{n+\ell} x)}{n} - \frac{1}{M_i}\sum_{n=1}^{N_i}  \frac{1}{L} \sum_{\ell=0}^{L-1} \frac{\mu(n+\ell) f(T^{\ell} x_{j_n})}{n} \right|\\
		&\le  \frac{1}{M_i}\sum_{n=1}^{N_i} \frac{1}{L} \sum_{\ell=0}^{L-1} \frac{\left|f(T^{\ell} (T^n x) )-f(T^\ell x_{j_n})\right|}{n}\\
		&\le \frac{1}{M_i}\sum_{n\in [1,N_i]\setminus E} \frac{1}{L} \sum_{\ell=0}^{L-1} \frac{\left|f(T^{\ell} (T^n x) )-f(T^\ell x_{j_n})\right|}{n}\\
		&\hskip3cm+\frac{1}{M_i}\sum_{n\in E\cap [1,N_i]} \frac{1}{L} \sum_{\ell=0}^{L-1} \frac{\left|f(T^{\ell} (T^n x) )-f(T^\ell x_{j_n})\right|}{n}\\
		&<\frac{2}{M_i}\sum_{n\in [1,N_i]\setminus E}\frac{1}{n}+\frac{3\epsilon}{M_i}\sum_{n\in E\cap [1,N_i]}\frac{1}{n}     \ \ \ \text{ (by \eqref{ineq-33})}\\
		&< \frac{2}{M_i}\sum_{n\in [1,N_i]\setminus E}\frac{1}{n}+3\epsilon.
	\end{align*}
	
	Combining this inequality with \eqref{ineq-2}, when $i$ is large enough,
	{\small	\begin{align}\label{ineq-4}
			\begin{split}
				&\hskip0.5cm  \left|\frac{1}{M_i}\sum_{n=1}^{N_i} \frac{1}{L} \sum_{\ell=0}^{L-1}
\frac{\mu(n+\ell) f(T^{n+\ell} x)}{n} - \frac{1}{M_i}\sum_{n=1}^{N_i}  \frac{1}{L} \sum_{\ell=0}^{L-1} \frac{\mu(n+\ell) f(T^{\ell} x_{j_n})}{n} \right| \\
				&< 5\epsilon.
			\end{split}
	\end{align}}
	So the lemma follows by \eqref{ineq-1} and \eqref{ineq-4}. This ends the proof of Lemma \ref{A-lem-2}.
\end{proof}

Now we proceed to show  Lemma \ref{A-lem-3}.
\begin{lem} \label{A-lem-3}For all sufficiently large $i$,
	$$\left|\frac{1}{M_i}\sum_{n=1}^{N_i}  \frac{1}{L} \sum_{\ell=0}^{L-1} \frac{\mu(n+\ell) f(T^\ell x_{j_n})}{n} \right|< 2\epsilon.$$
\end{lem}
\begin{proof} By Cauchy's inequality
	\begin{align*}
		&\hskip0.5cm \left|\frac{1}{M_i}\sum_{n=1}^{N_i}  \frac{1}{L} \sum_{\ell=0}^{L-1}\frac{ \mu(n+\ell) f(T^\ell x_{j_n}) }{n} \right|^2\\
		&\le \frac{1}{M_i}\sum_{n=1}^{N_i}\frac{1}{n}\left |\frac{1}{L} \sum_{\ell=0}^{L-1} \mu(n+\ell) f(T^\ell x_{j_n}) \right|^2\\
		&\le \frac{1}{M_i}\sum_{n=1}^{N_i}\frac{1}{n}\sum_{j=1}^m\left |\frac{1}{L} \sum_{\ell=0}^{L-1} \mu(n+\ell) f(T^\ell x_{j}) \right|^2\\
		&\le \frac{1}{L^2}\sum_{j=1}^m\sum_{\ell_1=0}^{L-1}\sum_{\ell_2=0}^{L-1}\frac{f(T^{\ell_1}x_j)
\overline{f}(T^{\ell_2}x_j)}{M_i}\sum_{n=1}^{N_i}\frac{\mu(n+\ell_1)\mu(n+\ell_2)}{n}.
	\end{align*}	
	Note that  $M_i\approx\log N_i$. Since the 2-terms logarithmic Chowla's conjecture holds \cite{Tao}, we have
	$$\lim_{i\rightarrow \infty} \frac{1}{M_i} \sum_{n=1}^{N_i}\frac{\mu(n+\ell_1)\mu(n+\ell_2)}{n}=0$$
	for any $0\le \ell_1\neq \ell_2\le L-1$. Combining this equality with the fact that
$\max_{x\in X}|f(x)|\le1$, one has that for sufficiently large $i$,
	\begin{align*}
		&\hskip0.5cm \left|\frac{1}{M_i}\sum_{n=1}^{N_i}  \frac{1}{L} \sum_{\ell=0}^{L-1}\frac{ \mu(n+\ell) f(T^\ell x_{j_n}) }{n} \right|^2\\
		&<\epsilon + \sum_{j=1}^m\frac{1}{L^2}\sum_{\ell=0}^{L-1}\sum_{j=1}^m\frac{|f(T^{\ell}x_j)
\overline{f}(T^{\ell}x_j)|}{M_i}\sum_{n=1}^{N_i}\frac{|\mu(n+\ell)\mu(n+\ell)|}{n}\\
		&\le \epsilon+\frac{m}{L^2}\sum_{l=0}^{L-1}\frac{1}{M_i}\sum_{n=1}^{N_i}\frac{1}{n}\\
		&=\epsilon+\frac{m}{L}\\
		&\overset{\eqref{2017-2}}< 2\epsilon.
	\end{align*}
	This ends the proof of Lemma \ref{A-lem-3}.
\end{proof}


\section{Proof of Theorem \ref{thm-5}}\label{Appendix-B}
In this appendix, we prove Theorem \ref{thm-5}. As in the proof of Theorem \ref{thm-2}, we let $p$ be the zero of $\mathbb{C}$.
For a sequence $y\in \left(\mathbb{T}\cup\{p\}\right)^\mathbb{Z}$, we put
$X_y=\overline{\{\sigma^ny:n\in\mathbb{Z}\}}$, where $\sigma$ is the left shift.
To this aim, we give a lemma firstly.
\begin{lem}\label{A-3-7} If there exist a non-empty compact subset $C$ of $[0,1]$
and  $\beta\in\mathbb{R}$ such that \begin{align}\label{A-00000}\limsup_{H\to+\infty}\limsup_{N\to+\infty}
\mathbb{E}_{n\le N}\max\{\sup_{\alpha\in C}Re\left(e(\beta)\mathbb{E}_{h\le H}\mu(n+h)e(h\alpha)\right),0\}>0,
	\end{align}
	then there is $y\in \left(\mathbb{T}\cup\{p\}\right)^{\mathbb{Z}}$ such that $(y,C)$ meets  Property $(*)$ in Lemma \ref{lem} and
	\begin{align}\label{eq-30}\limsup_{N\to\infty}|\mathbb{E}_{n\le N}
		\mu (n)\widetilde F(\sigma^ny)|>0,
	\end{align}
	where $\widetilde F: X_y\to \mathbb{C}$ is a continuous function defined by $\widetilde F( z)=z(0)$ if $z(0)\in \mathbb{T}$ and $0$ if $z(0)=p$.
\end{lem}
\begin{proof} It follows by a similar arguments of the proof of Lemma \ref{lem}.
\end{proof}	
Now we are going to prove Theorem \ref{thm-5}.
\begin{proof}[Proof of Theorem \ref{thm-5}] Assume the contrary that Theorem \ref{thm-5} doesn't hold. Then there
exists a non-empty compact subset $C$ of $[0,1]$ such that $Dim_P C=0$ and
	$$\limsup_{H\to+\infty}\limsup_{N\to+\infty}\mathbb{E}_{n\le N}\sup_{\alpha\in C}|\mathbb{E}_{h\le H}\mu(n+h)e(h\alpha)|>0.$$
	Thus, there is $\beta\in \{0,\frac{1}{4},\frac{2}{4},\frac{3}{4}\}$ with
	\begin{align*}
		\limsup_{H\to+\infty}\limsup_{N\to+\infty}\mathbb{E}_{n\le N}\max\{ \sup_{\alpha\in C}
Re\left( e(\beta)\mathbb{E}_{h\le H}\mu(n+h)e(h\alpha)\right),0\}>0.
	\end{align*}
	By Lemma \ref{A-3-7}, there is  $y\in \left(\mathbb{T}\cup\{p\}\right)^{\mathbb{Z}}$ such
that $(y,C)$ meets  Property $(*)$ in Lemma \ref{lem} and
	\begin{align}\label{A-eq-333}\limsup_{N\to\infty}|\mathbb{E}_{n\le N}
		\mu (n)\widetilde F(\sigma^ny)|>0,	\end{align}
	where $\widetilde F: X_y\to \mathbb{R}$ is a continuous function defined by $\widetilde F( z)=z(0)$
if $z(0)\in \mathbb{T}$ and $0$ if $z(0)=p$. By Lemma \ref{lem-1}, the t.d.s. $(X_y,\sigma)$ satisfies
$$\liminf_{n\rightarrow +\infty} \frac{S_n(d,\sigma ,\rho,\epsilon)}{n^\tau}=0 \text{	
for any } \epsilon>0, \tau>0\text{ and }\rho\in \mathcal{M}(X_y,\sigma),$$
	since $Dim_P C=0$.
	Using the result of \cite{HWY}, one has
	$$\limsup_{N\to\infty}|\mathbb{E}_{n\le N}
	\mu (n)\widetilde F(\sigma^ny)|=0.$$ This conflicts with \eqref{A-eq-333} and the
theorem follows. This ends the proof of Theorem \ref{thm-5}.
\end{proof}
\section*{Acknowledgement}
W. Huang was partially supported by NNSF of China (11431012,11731003),
L. Xu was partially supported by NNSF of China (11801538, 11871188)
and X. Ye was partially supported by NNSF of China (11431012).	

\end{document}